\documentclass{siamltex}
\usepackage{graphicx, amsfonts, eucal}
\pdfoutput=1

\newtheorem{Proposition}{Proposition}

\newtheorem{Lemma}{Lemma}
\newtheorem{Cor}{Corollary}

\newcommand{\rmk}{\bigskip \par \noindent {\it Remark: }}

\newcommand{\cA}{\mathcal{A}}                 
\newcommand{\cB}{\mathcal{B}}
\newcommand{\cC}{\mathcal{C}}

\newcommand{\cU}{\mathcal{U}}

\newcommand{\proj}[1]{\bbP_{\!#1}}
\newcommand{\projh}[1]{\proj{{\rm hor}}}

\newcommand{\R}{{\mathbb R}}

\newcommand{\lp}{\left (}
\newcommand{\rp}{\right )}
\newcommand{\la}{\left \langle}
\newcommand{\ra}{\right \rangle}

\newcommand{\lcb}{\left \{}
\newcommand{\rcb}{\right \}}

\newcommand{\sands}{\qquad \mbox{and} \qquad}
\newcommand{\half}{{\textstyle {1 \over 2}}}

\newcommand{\smallfrac}[2]{{\textstyle {#1 \over #2}}}

\newcommand{\rmkend}{ \medskip}

\newcommand{\bbP}{{\mathbb P}}
\newcommand{\ddt}{\smallfrac{d\ }{dt}}

\newcommand{\beq}[1]{\begin{equation}\label{#1}}
\newcommand{\eeq}{\end{equation}}
\newcommand{\beqa}{\begin{eqnarray*}}
\newcommand{\eeqa}{\end{eqnarray*}}
\newcommand{\beqan}{\begin{eqnarray}}
\newcommand{\eeqan}{\end{eqnarray}}
\newcommand{\Prop}[1]{\acapon\begin{Proposition}\label{#1}}
\newcommand{\eProp}{\end{Proposition}}
\newcommand{\Lem}[1]{\acapon\begin{Lemma}\label{#1}}
\newcommand{\eLem}{\end{Lemma}}

\newcommand{\prfend}{\ \vrule height6pt width6pt depth0pt \medskip}

\newcommand{\Id}{\hbox{{\rm 1}\kern-3.8pt \elevenrm1}}

\newcommand{\lamu}{{\lambda}}
\newcommand{\lamuf}{{\lambda_{\!f}}}

\newcommand{\sfin}{s_{\! f}} 
\newcommand{\tfin}{t_{\! f}}
\newcommand{\xfin}{x_{\! f}}

\newcommand{\tilC}{\widetilde C}
\newcommand{\tilchi}{\widetilde \chi}
\newcommand{\sgn}{\mbox{sgn}\,}
\newcommand{\pd}[2]{\smallfrac {\partial #1}{\partial #2}}
\newcommand{\dddt}{\smallfrac {d^2 \ }{d t^2}}

\begin{document}

\title{Optimal control with moderation incentives}
\author{Debra Lewis\thanks{Mathematics Department, University of California, Santa Cruz, 
Santa Cruz, CA 95064. {\tt lewis@ucsc.edu}. Supported by NSF DMS-0405610}}

\maketitle

\begin{abstract}
A purely state-dependent cost function can be modified by introducing a control-dependent
term rewarding submaximal control utilization. A moderation incentive is identically zero on
the boundary of the admissible control region and non-negative on the interior; it is bounded
above by the infimum of the state-dependent cost function, so that the instantaneous total 
cost is always non-negative. The conservation law determined by the Maximum Principle, in 
combination with the condition that the moderation incentive equal zero on the boundary of the 
admissible control region, plays a crucial role in the analysis; in some situations, the initial and final values of the auxiliary variable are uniquely determined by the condition that the conserved quantity
equal zero along a solution of the arbitrary duration synthesis problem. Use of an alternate system of
evolution equations, parametrized by the auxiliary variable, for one-degree of freedom controlled acceleration systems, can significantly simplify numerical searches for solutions of the arbitrary 
duration synthesis problem. A one-parameter family of `elliptical' moderation incentives 
is introduced; the behavior of the well-known quadratic control cost and its incentive analog is 
compared to that of the elliptical incentives in two simple controlled acceleration examples. 
The elliptical incentives yield smooth solutions with controls remaining in the interior 
of the admissible region, while the quadratic incentive allows piecewise smooth solutions with
controls moving on and off the boundary of the admissible region; in these examples,
the arbitrary duration synthesis problem for the traditional quadratic control cost has no 
solution---the total cost is a monotonically decreasing function of the duration.
\end{abstract}

\section{Introduction}
Optimal control problems typically involve constraints on both the state variables and 
control. For example, (bio)mechanical systems cannot generate or withstand 
arbitrarily large forces or accelerations. For some cost functions, 
trajectories approaching the boundary of the admissible region are so extravagant that
the boundary can safely be left out of the mathematical model. However, when a task is to be executed 
as quickly as possible, the bounds on the possible play a crucial role in determining the optimal 
process---cost considerations would drive the solution outside the admissible region if these bounds
were not explicit imposed. 

If the relevant constraints are explicitly incorporated in the state space and admissible control 
region, the cost function for a time minimization problem is constant. Given the degeneracy of the 
constant cost function, the optimal control values are sought on the boundary of the admissible
controls set. In some situations of interest, geometric optimization and integration 
methods can be used (see, e.g., \cite{LO}) to work directly on the boundary. If geometric methods
are not available or desirable, penalty functions can be used to construct algorithms on an ambient
vector space that respect the boundary due to the prohibitive (possibly infinite) expense of crossing
it (see, e.g., \cite{BG, BB}, and references therein). In many situations, particularly in biological 
models, a close approach to the boundary of the 
admissible region is undesirable---stresses on joints, muscles, and bones are severe near the 
breaking or tearing points of these structures---but sometimes justified.
An animal may be willing to push itself to its physical limits to escape a 
high-risk situation; many machines are engineered to execute certain tasks rapidly, even if this 
involves high energy consumption and rapid wear of mechanical parts. In such cases, selection of
an appropriate penalty function is essential, as too severe a penalty will yield overly conservative
solutions. 

An important class of  optimal control problems can be interpreted as modified time minimization
problems, in which certain states are more costly than others. When modeling a conscious agent, the unit cost function 
of a traditional  time minimization problem can be interpreted as representing a uniform stress and/or 
risk  throughout the task,  while a modified time minimization cost function models a combination of 
instantaneous stresses and risks that explicitly depend on the current state. This formulation may be 
more natural than an admissible/inadmissible dichotomy, particularly for biological systems. For example, consider the Kane-Scher model of
the classic `falling cat' problem, in which a cat is suspended upside down and then released;  a typical cat can right itself without net angular momentum from heights of approximately one meter.
(See, e.g.,  \cite{Marey, KS, Mont93, LO}.)  Kane and Scher \cite{KS} proposed a two rigid body model of a cat; to 
eliminate the mechanically efficient but biologically unacceptable solution in which the front and 
back halves counter-rotate, resulting in a $360^\circ$ twist in the `cat', Kane and Scher imposed
a no-twist condition in their model. However, actual cats can and do significantly 
twist their bodies; replacing the no-twist condition with a deformation-dependent term in the cost 
function that discourages excessively large relative motions allows more realistic 
motions.  We will refer to optimal control problems with cost functions depending 
only on the state variables as modified time minimization problems. 

Given a modified time minimization problem, we are interested in modifying the cost 
function to take the control effort into account. We assume here that a cost function modeling the  
do-or-die, `whatever it takes' approach is known, and construct a new cost function by 
subtracting a control-dependent term. Our approach is to regard this  term  not as a penalty or cost, 
but a deduction rewarding submaximal control efforts. Hence we specify that this term equal zero 
on the boundary of the admissible control set and be bounded above by the minimum of the original 
cost function, so that the total instantaneous cost function remains non-negative. 
We can construct parametrized families of such functions and adjust the urgency of the task by 
adjusting the parameter. The incentive function may allow controls to move on and off the boundary 
of the admissible control region, or may approach zero  sufficiently rapidly as the control approaches 
the boundary that controls starting in the interior of the admissible control region will remain there
throughout the maneuver. 

The notion of a moderation incentive can guide the modification of familiar cost functions. As we shall see in 
the examples, simply modifying the cost  function by a constant can make the difference between
the existence and absence of solutions of the arbitrary duration synthesis problem. The 
cost functions we use here to illustrate this property are quadratic in the control.
Quadratic control cost (QCC) minimization, with cost functions of the form $C(x, u) =  Q_x(u)$ for 
some smooth family of quadratic forms $Q_x$ determined by an inner product or Riemannian metric, 
has played an important role in geometric optimal control 
theory. (See, e.g.,  \cite{BW},  \cite{Bloch}, \cite{Mont92}, and references therein.) 
If the admissible control region is unbounded, QCC functions yield relatively 
simple evolution equations: if the state space is a subset of a Riemannian manifold $M$, the space of
admissible controls has full rank, $\dot x = u$,  and $C(x, u) = \half |u|_x^2$, then the traces in $M$ of 
the optimal trajectories are geodesics. If the control $u$ is constrained to lie in a distribution of less 
than full rank, the corresponding QQC problem leads to sub-Riemannian geometry.
(See, e.g. \cite{Montqq} or \cite{BW}.)  Thus many existing results from geometric mechanics and (sub-)Riemannian geometry can be utilized in the analysis of simple QCC control problems.  QCC
optimization sometimes follows a `the slower, the better' strategy: in some important  QCC problems 
the total cost is a decreasing function of the maneuver duration; hence there is no optimal solution 
of the arbitrary duration QCC problem.
If there is a range of durations $[T_{\rm min}, \infty)$ for which unique specified duration QQC solutions 
exist, then the QCC trajectory of duration $T_{\rm min}$ may be of interest as the fastest of the slow.  
However, it is unclear in what sense these trajectories are optimal. 
Modifying the quadratic control cost by a constant, so as to satisfy the condition that the moderation 
incentive equal zero on the boundary of the admissible control region, can yield arbitrary duration 
synthesis problems for which unique solutions {\it do} exist in situations where the QCC
function (nonzero on the boundary of the sphere) lacks such solutions. 


We introduce a one-parameter family of `elliptical' moderation incentives
$\tilC_\mu: [0, 1] \to [0, \mu]$, $\mu \in (0, 1]$, by
\beq{tilC_mu}
\tilC_\mu( s) := \mu \sqrt{1 - s^2}.  
\eeq
(The graph of $\mu \, \sqrt{1 - s^2}$, $0 \leq s \leq 1$, is a segment
of an ellipse of eccentricity $\mu$.) $\tilC_0$ is the trivial incentive associated to the unmoderated
modified time minimization problem. For $\mu \in (0, 1]$, the control values determined by
$\tilC_\mu$ always lie in the interior of the unit ball; if the unmoderated cost function is smooth,
the state variables and control will also be smooth. 
(In contrast, some of the solutions for the quadratic incentive and QCC we find in the examples
are only piecewise smooth, moving on and off the unit sphere.) However, the penalty imposed by 
$\tilC_\mu$ is not prohibitive; as we shall see in the examples, we can come arbitrarily close to the 
control region boundary by adjusting $\mu$. The elliptical incentives have some
simple properties that make them particularly convenient to work with in certain kinds of analytic
and numerical calculations. Finally, in the examples treated here, there
are some qualitative resemblances between the trajectories determined by the quadratic incentive
$\tilC_{\rm q}$ and the elliptical incentives for values of $\mu$ near 1. 

We make several simplifying assumptions in the present work. We assume that the admissible
control region is the Euclidean unit ball and our incentives are nonincreasing functions of 
the magnitude of the control. We restrict our attention to problems in which the state variables consist of
a position vector in $\R^n$ and its first $k - 1$ derivatives; the $k$-th derivative is fully controlled
and the unmoderated cost function depends only on the position. These assumptions are not
central to the formulation of the moderated problem, but they lead to particularly simple
expressions in some key constructions. More general control systems, with more complex 
admissible control regions (including ones determined in part by the state variables) and more
general controls, will be considered in future work.  	

We consider two examples that illustrate some of the key features of the moderated
control problems and suggest directions of future research. The first example is a very simple 
one-dimensional controlled acceleration problem: a particle at rest at one position is to be moved 
a unit distance by controlling the acceleration; the initial and terminal velocities are zero. This
classic starter problem is treated in \cite{PBGM, Kirk}, and other texts. The well-known time 
minimizing solution is the `bang-bang' solution, with acceleration equal to 1 for the first half of the maneuver and $-1$ for the second half; the solution of the arbitrary duration problem with
quadratic incentive has  linear acceleration;  the solutions for the elliptical incentives have 
smooth accelerations approaching the bang-bang solution as the moderation parameter approaches
zero, and approaching the QCC solution as the parameter approaches one.

The second example is a generalization of the first, with a position penalty added to the control
cost. The position penalty is monotonically decreasing and equals zero at the destination. 
This example can be interpreted as a very simple model of `spooking' (flight reaction), in which
the position penalty models aversion to a localized stimulus and the destination is the position
at which the animal first feels entirely safe or comfortable.  The reflectional symmetry seen in the
first example is broken: all of the cost functions studied here, with the exception of the trivial 
moderation incentive, yield asymmetric solutions, with relatively strong initial accelerations and
relatively weak decelerations. The quadratic incentive yields only piecewise smooth solutions,
while the elliptical incentive solutions are smooth for all nonzero values of the moderation 
parameter. For small-to-middling values of the moderation 
parameter $\mu$, the solutions for the elliptical incentives show little response 
to the intensity of the position penalty---the solutions remain close to the corresponding solutions for 
the corresponding problem without a position penalty even when the position penalty is high. 
Roughly speaking, if little or no incentive to take it easy is added to a time-pressured task, the optimal strategy is to get it all over with (almost) as quickly as possible; there's little room 
for modification of the strategy if additional discomfort or risk  is introduced.  On the other hand, if there's
a significant reward for moderate effort, the strategy in the absence of a position penalty will be to
take it slowly, and the introduction of some variable risk or discomfort can yield dramatic speed-ups
in overall execution times, as well as significant variations in control magnitudes.The classic quadratic control cost (QCC) function, which is nonzero on the boundary of the 
admissible control region, determines a total cost that is a monotonically decreasing function of the 
maneuver duration; as the specified duration is increased, the solutions perform an increasing 
number of oscillations about the destination before coming to rest.

\section{The $k$-th order moderated synthesis problem}

We apply Pontryagin's Maximum Principle to a special class of optimal control problems
that illustrate some of the features of moderated control problems, but are relatively easily analysed.
We focus  on the $k$-th order evolution equation $x^{(k)} = u$, $x: [0, \tfin] \to \R^n$, 
with specified initial and final values of the $x^{(j)}$, $j = 1, \ldots, k - 1$, and unmoderated 
cost function depending only on the position $x$, not the derivatives of $x$. We further simplify 
the analysis by assuming that the admissible control region is the unit ball in $\R^n$ and that 
the cost depends on the control only through its norm.

Consider a control problem with state variable $z \in \R^m$, control $u \in \cU \subset \R^k$,
evolution equation $\dot z = V(z, u)$, boundary conditions $z(0)  = z_0$ and $z(\tfin) = z_{\! f}$,
and cost function $C: \R^m \times \cU \to \R$. Assume that both $V$ and $C$ are continuous,
with continuous derivatives with respect to
the state variable $z$; let $\cA$ denote the set of triplets $(z, u, \tfin)$, such that $\tfin > 0$,
$(z, u): [0, \tfin] \to \R^m \times \cU$ satisfies the evolution equation and boundary equations, $z$ 
is continuous with piecewise continuous derivative $\dot z$, and $u$ is piecewise continuous.
Pontryagin's Maximum Principle states that if $(z, u, t) \in \cA$ minimizes the total cost over $\cA$, i.e.
\[
\int_0^{\tfin} C(z(t), u(t)) dt = \min_{(\tilde z, \tilde u, \tilde \tfin) \in \cA}
\int_0^{\tilde \tfin} C(\tilde z(t), \tilde u(t)) dt,
\]
then there is a continuous curve $\psi: [0, \tfin] \to \R^m$ and constant $\phi \geq 0$ such that
\beq{Pont_ham}
\dot z = \pd {H_\phi} \psi(z, \psi, u), \qquad 
\dot \psi = - \pd {H_\phi} z(z, \psi, u),
\sands
H_\phi(z, \psi, u) = \max_{\upsilon \in \cU} H_\phi(z, \psi, \upsilon)
\eeq
for $H_\phi(z, \psi,  u) :=  \la \psi, V(z, u) \ra - \phi \, C(z, u)$. The Hamiltonian $H_\phi$ is constant
along a curve satisfying (\ref{Pont_ham}); if the curve minimizes the curve on $\cA$, then constant
is zero. (See \cite{PBGM} for the precise statement and proof of the Maximum Principle.) 

Pontryagin's conditions are necessary, but not sufficient, for optimality; their appeal lies in their
constructive nature: known results and techniques for boundary value problems and Hamiltonian systems can be used in constructing the triplets $(z, u, \tfin)$ satisfying Pontryagin's conditions.
This construction is referred to as the synthesis problem in \cite{PBGM}. We will restrict our
attention to the synthesis problem, setting aside the rigorous analysis of actual optimality of the 
trajectories we obtain. It suffices to consider the 
cases $\phi = 1$ and $\phi = 0$, since we can rescale $\psi$ by $\phi \neq 0$. $H_0$ is clearly
independent of the cost function $C$; Hamilton's equations for $H_0$ equal those for a
constant cost function, used in determined the minimum-time admissible curves. Hence we
will focus on $H_1$, simply noting that when searching for the optimal solution, it is necessary
to consider the possibility that the total cost is minimized by the minimum duration trajectory. 

Here we consider control problems of the form $x^{(k)} = u$,  where 
$x: [0, \tfin] \to \R^n$ and
$u: [0, \tfin] \to \cB^n = \lcb u \in \R^n : |u| \leq 1 \rcb$, for some (possibly specified) 
$\tfin > 0$, with specified boundary values $x^{(j)}(0)$ and $x^{(j)}(\tfin)$, $j = 0, \ldots, k - 1$.
We restrict our attention to cost functions that are the difference of a term depending only on 
the position $x$ and a term depending only on the magnitude of the control. We assume that
the position-dependent term is bounded below (for simplicity, we take the bound to be 1), and
the control-dependent term has range contained within [0,1]; thus the instantaneous cost is
non-negative. Given $C \in \cC^1(\R^n, [1, \infty))$ and $\tilC \in \cC^0([0, 1], [0, 1])$, 
we seek $x$ with piecewise continuous $k$-th derivative and piecewise continuous 
$u$ minimizing the total cost
\[
\int_0^{\tfin} \lp C(x(t)) - \tilC(|u(t)|) \rp dt
\]
over all such state/control variable pairs.

To apply the  Pontryagin Maximum Principle to our control problem, we first
convert the $k$-th order evolution equation $x^{(k)} = u$ into a first order system of ODEs by
introducing the auxilliary state variables $d_j := x^{(j)}$, 
$j = 1, \ldots, k - 1$, and setting $z = (x, d_1, \ldots, d_{k - 1}) \in (\R^n)^k \approx \R^{n \, k}$. The
resulting first order evolution equation is
\beq{first_ord_sys}
\dot z = V(z, u) := (d_1, \ldots, d_{k - 1}, u).
\eeq
If we let $\psi = (\kappa_1, \ldots, \kappa_{k - 1}, \lamu) \in (\R^n)^k \approx \R^{n \, k}$, then 
$H_1$ is equivalent to the Hamiltonian $H: (\R^n)^{2 \, k + 1} \to \R$ given by
\beq{big_H}
H(x, d_1, \ldots, d_{k - 1}, \kappa_1, \ldots, \kappa_{k - 1}, \lamu, u) 
:= \sum_{j = 1}^{k - 1} \la \kappa_j, d_j \ra + \la \lamu, u \ra + \tilC(|u|) - C(x).
\eeq

The control $u$ is chosen at each time $t$ so as to maximize the Hamiltonian. 
Since (\ref{big_H}) depends on $u$ and $\lamu$ only through the term
$\la \lamu, u \ra + \tilC(|u|)$, the optimal value of $u$ satisfies $|\lamu| \, u = |u| \, \lamu$, and
\beq{bigH}
H(x, d_1, \ldots, d_{k - 1}, \kappa_1, \ldots, \kappa_{k - 1}, \lamu, u) 
= \max_{\upsilon \in \cB^n} H(x, d_1, \ldots, d_{k - 1}, \kappa_1, \ldots, \kappa_{k - 1}, \lamu, \upsilon)
\eeq
if and only if
\beq{max_cond}
|u| \, |\lamu| + \tilC(|u|) = \chi(\lamu) := \max_{0 \leq \sigma \leq 1} \sigma \, |\lamu| + \tilC(\sigma).
\eeq
If $\tilC(0) > \tilC(s)$ for all $s \in (0, 1]$, then $\lamu = 0$ implies $u = 0$; if
$\tilC$ achieves its maximum at any point other than the origin, $u$ is
not uniquely determined when $\lamu = 0$. (In most of the cases considered 
here, $u$ is uniquely determined at $\lamu = 0$ but this is not true for the 
time minimization problem.) 

Hamilton's equations 
\[
\begin{array}{ll}
{\displaystyle \dot x = \pd H {\kappa_1} \ = \ d_1} 
& \qquad {\displaystyle \dot \kappa_1 = - \pd H x \ = \ \nabla C(x)} \smallskip \\
{\displaystyle \dot d_j = \pd H {\kappa_{j + 1}} \ = \ d_{j + 1} }
& \qquad {\displaystyle \dot \kappa_j = - \pd H {d_{j - 1}} \ = \ - \kappa_{j - 1},
	\qquad j = 2, \ldots, k - 1} \smallskip \\
{\displaystyle \dot d_{k - 1} = \pd H \lamu \ = \ u} 
& \qquad {\displaystyle \dot \lamu = - \pd H {d_{k - 1}} \ = \ - \kappa_{k - 1}}
\end{array}
\]
for the Hamiltonian (\ref{bigH}) are equivalent to $x^{(j)} = d_j$ and $\lamu^{(j)} = (-1)^j \kappa_{k - j}$, 
$j = 1, \ldots, k - 1$, $x^{(k)} = u$, and $\lamu^{(k)}  = (-1)^{k - 1} \nabla C(x)$.
Inserting these expressions into the Hamiltonian (\ref{big_H}) yields
\beq{cons_law}
\chi(\lamu)  - C(x) + \sum_{j = 1}^{k - 1} (-1)^j \la x^{(k - j)}, \lamu^{(j)} \ra
\eeq
Pontryagin's Maximum Principle implies that (\ref{cons_law}) is constant along 
a curve $(x, \lamu, u)$ satisfying Hamilton's equations and maximizing the
Hamiltonian.

\begin{definition}
$x \in \cC^{k - 1}([0, \tfin], \R^m)$, with piecewise continuous $k$-th 
derivative, satisfying the given boundary conditions on $x^{(j)}(0)$ and
$x^{(j)}(\tfin)$, $j = 0, \ldots, k - 1$, is a solution of the 
{\it $k$-th order synthesis problem} if there exists $\lamu:[0, \tfin] \to \R^m$ 
satisfying
\beq{pre_opt}
\lamu^{(k)}  = (-1)^{k - 1} \nabla C(x)
\sands
|x^{(k)}| \, |\lamu| + \tilC(|x^{(k)}|) = \chi(\lamu)
\eeq
for $0 \leq t \leq \tfin$. If $x^{(k)}$ is discontinuous at $t_*$, then $x^{(k)}(t_*)$ agrees with 
the left hand limit, i.e. $x^{(k)}(t_*) = \lim_{t \to t_*^-}x^{(k)}(t)$. 

If, in addition, (\ref{cons_law}) equals zero along the curve $(x, \lamu)$, $x$ is a solution 
of the {\it $k$-th order synthesis problem of arbitrary duration}.
\end{definition}

We follow the convention of \cite{PBGM} in specifying that $x^{(k)}$ is continuous to the left
at a discontinuity; one could as well choose the right hand limit.

We introduce a class of functions $\tilC$ for which the optimal value of the 
control $u$ is explicitly given as the gradient of a function of $\lamu$ when $\lamu \neq 0$,
determining a system of $k$-th order system of ODEs for $x$ and $\lamu$. The functions
equal zero on the sphere $S^{n - 1}$ bounding the admissible control region; the instantaneous
control cost equals the position-dependent term at peak control values.

\begin{definition}
If $\tilC \in \cC^0([0, 1], [0, 1])$ is differentiable on $(0, 1)$, $\tilC(1) = 0$, and there is a unique 
non-decreasing function $\sigma \in \cC^0(\R^+, (0, 1])$, differentiable on 
$\sigma^{-1}(0, 1)$, satisfying
\beq{sig_cond}
\sigma(s) \, s + \tilC(\sigma(s)) = \tilchi(s) := \max_{0 \leq \sigma \leq 1} \sigma \, s + \tilC(\sigma),
\eeq
we say that $\tilC$ is a {\it moderation incentive}, with {\it moderation potential}
$\chi: \R^n \to [0, \infty)$ given by $\chi(\lamu) := \tilchi(|\lamu|)$.
\end{definition}

\begin{Lemma}
\label{sig_lemma}
If $\tilC$ is a moderation incentive,  then $\tilchi$ is continuously differentiable
and strictly increasing on $\R^+$, with $\tilchi' = \sigma$. 
\end{Lemma}

\begin{proof}
If $\sigma(s) < 1$, then $\sigma(s)$ is a critical point of
$\sigma \mapsto \sigma \, s - \tilC(\sigma)$, and hence 
$s = \tilC'(\sigma(s))$. In addition, $\tilchi$ is differentiable at $s$,
with derivative
\[
\tilchi'(s) = \sigma(s) + (s - \tilC'(\sigma(s)) \sigma'(s) = \sigma(s).
\]

If $\sigma^{-1}(1) \neq \emptyset$, then $\sigma^{-1}(1) = [s_*, \infty)$ for some
$s_*$, since $\sigma$ is non-decreasing. Since $\tilchi(s) = s - \tilC(1)$
for $s \geq s_*$, $\tilchi$ is clearly differentiable and satisfies
$\tilchi'(s) = 1 = \sigma(s)$ for $s > s_*$. Continuity of $\sigma$ and the
Mean Value Theorem imply that $\tilchi$ is differentiable at $s_*$, with
$\tilchi'(s_*) = 1 = \sigma(s_*)$.
\end{proof}

We focus our attention on a one-parameter family of moderation incentives, which includes the
trivial incentive $\tilC_0 \equiv 0$, and a quadratic polynomial moderation incentive that differs
from a kinetic energy term by a constant in the case of controlled velocity. 

\begin{Proposition}
The functions $\tilC_\mu(s) = \mu \sqrt{1 - s^2}$, $0 \leq \mu \leq 1$, are  moderation incentives, 
with moderation potentials
\beq{MTM_cp}
\chi_\mu(\lamu) = \sqrt{\mu^2 + |\lamu|^2}.
\eeq
$\chi_\mu \in \cC^1(\R^m, \R^+)$ if $\mu > 0$; $\chi_0$ is continuously differentiable 
everywhere except at the origin, where the equation 
determining the optimal control value is completely degenerate.

The quadratic polynomial $\tilC_{\rm q}(s) = \half \lp 1 - s^2 \rp$ is a moderation incentive 
function, with moderation potential 
\beq{QCC_cp}
\chi_{\rm q}(\lamu) = \lcb \begin{array}{ll}
\half \lp 1 + |\lamu|^2 \rp \qquad & |\lamu| \leq 1\\
|\lamu| & |\lamu| > 1
\end{array} \right .
\eeq
$\chi_{\rm q} \in \cC^1(\R^m, \R^+)$.
\end{Proposition}
\begin{proof}
If  $\mu > 0$, differentiating
\[
\sigma \,s + \tilC_\mu(\sigma) 
= \sigma \, s + \mu \sqrt{1 - \sigma^2} 
\]
with respect to $\sigma$ yields the criticality condition
$s = {\mu \, \sigma}/{\sqrt{1 - \sigma^2}}$,
with unique solution
\beq{opt_sig_mu}
\sigma = \frac s {\sqrt{\mu^2 + s^2}}.
\eeq
The inequality 
\[
\frac {s^2} {\sqrt{\mu^2 +s^2}} + \tilC_\mu \lp \frac {s} {\sqrt{\mu^2 +s^2}} \rp 
= \sqrt{\mu^2 +s^2}
> \mbox{max} \lcb \mu, s \rcb
= \mbox{max} \lcb \tilC_\mu(0), s + \tilC_\mu(1) \rcb
\]
for $\mu > 0$ and $s > 0$ implies that (\ref{opt_sig_mu}) is the optimal value of $\sigma$,
and hence $\chi_\mu$ is given by (\ref{MTM_cp}). $\sigma \, s + \tilC_0(s) = \sigma \, s$ achieves
its maximum $s = \sqrt{0^2 + s^2}$ at the boundary point $\sigma = 1$.

We now consider the quadratic polynomial $\tilC_{\rm q}$:
\[
s \, \sigma + \tilC_{\rm q}(\sigma) 
= s \, \sigma + \frac {1 - \sigma^2} 2
= 1 - \frac {(s - \sigma)^2} 2 
\] 
achieves its maximum on $[0, 1]$ at $\sigma = \mbox{min} \lcb s, 1 \rcb$.
\end{proof}

The evolution equations for $x$ and $\lamu$ in the synthesis problem associated to a
moderation incentive are a pair of $k$-th order skewed gradient equations.

\begin{Proposition}
\label{general}
Let $C \in \cC^1(\R^n, [1, \infty))$ and $\tilC$ be a moderation incentive , with 
moderation potential $\chi$.
$x: [0, \tfin] \to \R^n$ satisfying the given boundary conditions is a solution of the synthesis 
problem for the cost function $C(x) - \tilC(|u|)$ if and only if there exists 
$\lamu: [0, \tfin] \to \R^n$ satisfying
\beq{opt_gen_pair}
x^{(k)} = \nabla \chi(\lamu)
\sands
\lamu^{(k)} = (-1)^{k - 1} \nabla C(x).
\eeq
If, in addition, the conserved quantity (\ref{cons_law})
equals zero, $(x, \lamu)$ is a solution of the arbitrary duration synthesis problem.

If $\lamu(t_*) = 0$ at some time $t_*$ and $\nabla \chi(0)$ is undefined,  the first
equation in (\ref{opt_gen_pair}) is replaced with $x^{(k)} = \lim_{t \to t_*^-} \nabla \chi(\lamu(t))$.
\end{Proposition}
\begin{proof}
Assume that $x$ is a solution of the synthesis problem, with auxilliary function $\lamu$. 
If $\lamu \neq 0$, uniqueness of the maximizer $\sigma$  and Lemma \ref{sig_lemma} imply that 
\[
x^{(k)} = \frac {\sigma(|\lamu|)}{|\lamu|} \, \lamu  = \frac {\tilchi(|\lamu|)}{|\lamu|} \, \lamu
= \nabla \chi(\lamu).
\]
On the other hand, if $\lamu$ exists such that $(x, \lamu)$ satisfies (\ref{opt_gen_pair}), then
$x^{(k)} = \nabla \chi(\lamu)$, then the same argument shows that (\ref{pre_opt}) is satisfied.

Since $\tilchi$ is $\cC^1$ on $\R^+$, and hence $\chi$ is $\cC^1$ on $\R^n\backslash\lcb 0 \rcb$,
the left handed limit $\lim_{t \to t_*^-} \nabla \chi(\lamu(t))$ is well-defined when $\lamu(t_*) = 0$,
and equals $\lim_{t \to t_*^-} x^{(k)}(t)$. 
\end{proof}

If $\nabla \chi$ is invertible, then the first equation in (\ref{opt_gen_pair}) can be solved for $\lamu$.
For example, the elliptic moderation incentives have a 
potential with invertible gradient: 
\[
\nabla \chi_\mu(\lamu) = \frac {\lamu} {\chi_\mu(\lamu)},
\qquad \mbox{with} \qquad
(\nabla \chi_\mu)^{-1}(u) = \frac {\mu \, u} {1 - |u|^2}.
\]
In this case, (\ref{opt_gen_pair}) is equivalent to a $2k$-order ODE in $x$.

\begin{Cor}
\label{double_cor}
Let $C \in \cC^1(\R^n, [1, \infty))$ and $\tilC$ be a moderation incentive with moderation potential
$\chi$. If $\nabla \chi$ is defined everywhere and invertible, then $x: [0, \tfin] \to \R^n$ satisfying 
the given boundary conditions is a solution of the synthesis  problem for the cost function 
$C(x) - \tilC(|u|)$ if and only if 
\beq{double_order}
\frac {d^k \ }{dt^k} (\nabla \chi)^{-1} \lp x^{(k)} \rp = (-1)^{k - 1} \nabla C(x)
\eeq
for $0 \leq t \leq \tfin$. 
If, in addition, 
\beq{double_cons}
C(x) = \chi \lp  (\nabla \chi)^{-1} \lp x^{(k)} \rp \rp 
+ \sum_{j = 1}^{k - 1} (-1)^j \la x^{(k - j)}, \frac {d^j \ }{dt^j} (\nabla \chi)^{-1} \lp x^{(k)} \rp \ra,
\eeq
$x$ is a solution of the arbitrary duration synthesis problem.
\end{Cor}

\section{One dimensional controlled acceleration systems}

In the arbitrary duration problem the final time $\tfin$ is generally not known
a priori; if a closed form expression for the solutions of the
synthesis problem cannot be found,  an iterative numerical procedure 
may be needed be find the appropriate $\tfin$. In some situations it may be both
possible and desirable to reformulate the problem to avoid this difficulty.  As an
example, we present an approach suitable for a class of one dimensional controlled 
acceleration problems. 
 
If $\dot \lamu$ is known a priori to be nonzero, we can reparametrize the 
evolution equations and use the conservation of (\ref{cons_law}) to replace the pair
of autonomous second order of equations (\ref{opt_gen_pair}) with a pair of first order 
nonautonomous ODEs, with independent variable $\lamu$, 
and a subordinate first order equation relating $\lamu$ and $t$. The induced boundary 
conditions for this problem may be more convenient  than
those of the original synthesis problem.
As we shall show below, the pair of first order ODEs for $x$ and the
auxiliary variable $q$ can be formulated without a priori knowledge of
the behavior of $\lamu$; if a solution $(x, q)$ is found that satisfies the 
relevant equalities and inequalities,  it determines a solution of the synthesis
problem. 

\begin{Proposition}
\label{reparam_prop}
If 
\begin{enumerate}
\item
there are functions $r: [\lamu_0, \lamuf] \to \R$ and $q: [\lamu_0, \lamuf] \to \R^+$
satisfying the evolution equations
\beq{reparam}
q \, r'  + C \circ r - \chi = \mbox{\rm constant},
\sands
q' = - 2 \, C'(r),
\eeq
and the boundary conditions $r(\lamu_0) = x_0$, $r(\lamuf) = \xfin$, 
$\sqrt{q(\lamu_0)} \, r'(\lamu_0) = v_0 \, {\rm sgn} (\lamuf - \lamu_0)$, and 
$\sqrt{q(\lamuf)} \, r'(\lamuf) = v_{\!f} \, {\rm sgn} (\lamuf - \lamu_0)$
\item
the solution of the IVP
\beq{reparam_lam}
\dot \lamu = \sgn (\lamuf - \lamu_0) \sqrt{q(\lamu)}
\sands
\lamu(0) = \lamu_0
\eeq
passes through $\lamuf$ at some positive time $\tfin$,
\end{enumerate}
then $x = r \circ \lamu: [0, \tfin] \to \R^n$ is a solution of the one dimensional
controlled acceleration synthesis problem with boundary data
$x(0) = x_0$, $x(\tfin) = \xfin$, $\dot x(0) = v_0$, and $\dot x(\tfin) = v_{\!f}$.
If the constant in (\ref{reparam}) is zero, then $x$ is a solution of 
the arbitrary duration synthesis problem.
\end{Proposition}
\begin{proof}
Differentiation of $\dot \lamu = \sgn (\lamuf - \lamu_0) \sqrt{q}$ yields
\[
\ddot \lamu = \sgn (\lamuf - \lamu_0) \frac {\dot q}{2 \, \sqrt{q}}
= \frac {\dot q}{2 \, \dot \lamu}
= \frac {q'} 2
= - C'(r).
\]
Differentiation of the first equation in (\ref{reparam_lam})  yields
\[
0 = (q \, r' + C \circ r - \chi)' 
= q \, r'' + q' r' + (C \circ r)' - \chi'
= q \, r'' -  (C \circ r)' - \chi'.
\]
Hence
\[
\dddt (r \circ \lamu) 
= r'' (\dot \lamu)^2 + r' \, \ddot \lamu
= q \, r'' - (C \circ r)'
= \chi'.
\]
Thus $(x \circ \lamu, \lamu)$ satisfy the evolution equations (\ref{opt_gen_pair}). The boundary conditions in {\it (i)} guarantee that $r \circ \lamu$ satisfies the boundary conditions of the
synthesis problem. 

If the constant in the first equation in (\ref{reparam_lam}) is zero, then
\[
\chi = q \, r' + C \circ r 
= \dot \lamu^2 r' + C \circ r 
= \dot \lamu \, \dot r + C \circ r ,
\]
and hence (\ref{cons_law}) is zero.
\end{proof}

In the arbitrary duration case, the velocity initial condition is satisfied if and only if
\beq{reparam_bda}
\lcb \begin{array}{ll}
\chi(\lamu_0) = C(x_0) & \qquad v_0 = 0 \\
q(\lamu_0) = v_0^2 \quad \mbox{and} \quad x'(\lamu_0) \, v_0 (\lamu_0 - \lamuf) > 0 
& \qquad v_0 \neq 0
\end{array} \right . ;
\eeq
entirely analogous conditions hold for the terminal velocity.
For example, if the initial and terminal velocities are both zero and $\tilchi$ is one-to-one, 
then $|\lamu_0| = \tilchi^{-1}(C(x_0))$ and $|\lamuf| =  \tilchi^{-1}(C(\xfin))$.
\medskip

\rmk
If it is known a priori that a solution $x$ of the synthesis problem must satisfy $\dot x \neq 0$ 
for all $t$, we can reduce the fourth order system (\ref{double_order}) to a third order one 
by solving (\ref{double_cons}) for $\ddt \frac {\ddot x} {\sqrt{1 - \ddot x^2}}$, obtaining
\[
\mu \, \ddt \frac {\ddot x} {\sqrt{1 - {\ddot x}^2}} \dot x = \frac \mu {\sqrt{1 - {\ddot x}^2}} - C(x),
\]
and hence
\[
x^{(3)} = \frac {1 - \ddot x^2 } {\dot x} \lp 1 - \frac {C(x)} \mu \sqrt{1 - \ddot x^2} \rp.
\]
\prfend

\subsection{Warm-up example: constant cost controlled acceleration}
\label{warm_up}

\begin{figure}
\begin{center}
\includegraphics[width=2.5in]{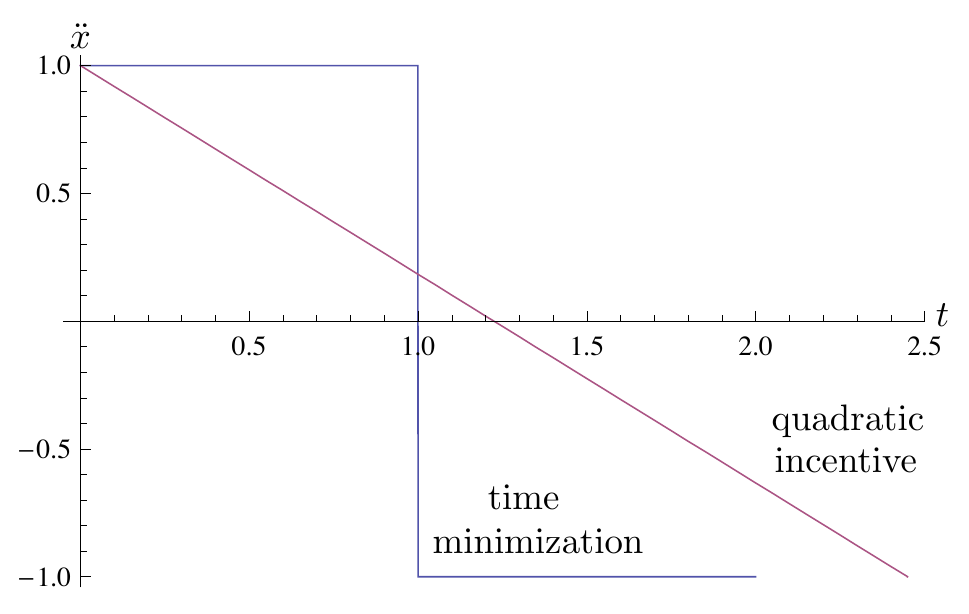} 
\quad
\includegraphics[width=2.5in]{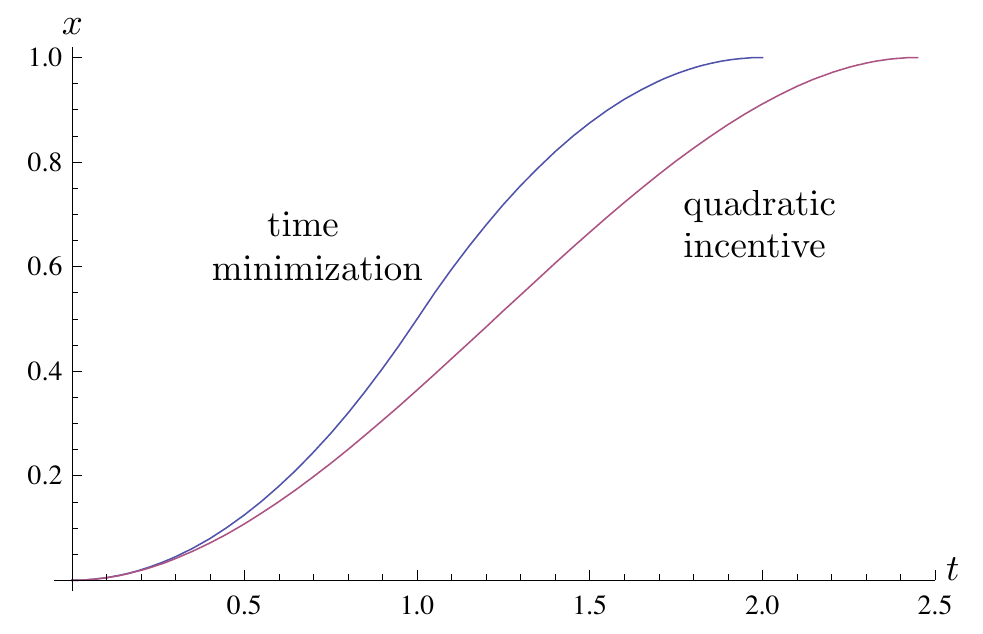} 
\end{center}
\caption{Solutions of the arbitrary duration synthesis problem for the time minimization problem, 
with trivial moderation incentive and duration $\tfin = 2$, and for the quadratic moderation 
incentive $\tilC_{\rm q}$, with $\tfin = \sqrt{6}$.
Left: acceleration $\ddot x(t)$; right: position $x(t)$.}
\end{figure}

As a simple illustrative example, we consider the one dimensional controlled acceleration problem
$\ddot x = u$, with boundary conditions 
\beq{bcs}
x(0) = 0, \qquad x(\tfin) = 1, \sands \dot x(0) = \dot x(\tfin) = 0
\eeq 
for some final time $\tfin > 0$, and constant cost function $C \equiv 1$.  
The trivial incentive version of this problem appears as the introductory 
example in \cite{PBGM}, and appears in several other control
texts. The `bang-bang' solution
\beq{bang}
x(t) = t + \half (|1 - t|(1 - t) - 1)
= \lcb \begin{array}{ll}
\frac {t^2} 2 & 0 \leq t \leq 1\medskip \\ -\frac {t^2} 2 + 2 \, t - 1  \qquad  & 1 < t \leq 2
\end{array} \right . ,
\eeq
with peak control effort throughout the maneuver, is optimal; see, e.g. \cite{Kirk}. 
Given a moderation incentive $\tilC: [0, 1] \to [0, 1]$, we 
seek a solution of the synthesis problem with boundary conditions
(\ref{bcs}) and unmoderated cost $C \equiv 1$.
Our treatment is a straightforward application of Proposition \ref{reparam_prop}. 
\begin{Proposition} 
\label{warmup_prop}
If $\tilC$ is a moderation incentive, then for any $\lamu_0 > 0$
\beq{first_eq}
q = 2 \,  \int_0^{\lamu_0} (\tilchi(\lamu_0) - \tilchi(s)) ds,
\qquad
r(\lamu) =  \frac 1 q \, \int_{\lamu}^{\lamu_0} (\tilchi(\lamu_0) - \tilchi(|s|)) ds,
\sands
\tfin = \frac {2 \, \lamu_0}{\sqrt{q}}
\eeq
determine a solution
\[
x(t) := r \lp \lamu_0 \lp 1 - 2 \, \frac  {t} {\tfin} \rp \rp, \qquad \qquad 0 \leq t \leq \tfin,
\]
of the synthesis problem with boundary conditions (\ref{bcs})
and cost function $1 - \tilC$. If $\tilchi(\lamu_0) = 1$, then $x$ is also 
a solution of the arbitrary duration synthesis problem. 
\end{Proposition}

\begin{proof}
The constant $q$ and function $r$ given by (\ref{first_eq}) clearly satisfy the 
evolution equations $q' = 0 = - 2 \, C'(x)$ and
\[
q \, r' + C \circ r - \chi = \tilchi(|\lamu|) - \tilchi(\lamu_0) + 1 - \chi(\lamu)
= 1 - \tilchi(\lamu_0)
\]
and boundary conditions  $r(\lamu_0) = 0$, $r(-\lamu_0) = 1$, and $r'(\pm \lamu_0) = 0$.
The auxiliary function $\lamu(t) := \lamu_0 \lp 1 - 2 \, t/\tfin \rp$ 
satisfies $\lamu(0) = \lamu_0$, $\lamu(\tfin) = - \lamu_0$
and $\dot \lamu = - \sqrt{q} = \sgn (\lamuf - \lamu_0) \sqrt{q}$. Hence Proposition
\ref{reparam_prop} implies that $x$ is a solution of the synthesis problem.
\end{proof}

We now apply Proposition \ref{warmup_prop} to the trivial, quadratic, and elliptical incentives.  
For the time minimization problem,  with trivial incentive $\tilC_{\rm tm} \equiv 0$, we have
$\tilchi_{\rm tm}(s) = s$, and hence
\[
r_{\rm tm}(\lamu) = \lamuf \lamu + \half \lp \lamuf^2 - |\lamu| \lamu \rp,
\sands
\tfin = \frac {2 \, \lamuf} { \sqrt{r_{\rm tm}(\lamuf)}} = 2.
\]
Thus we obtain the well-known `bang-bang' solution (\ref{bang})
for any positive $\lamuf$. The total cost is simply the duration, 2. 
Note that the condition 
$\tilchi(\lamuf) = 1$ is not necessary---the criticality (with respect to duration) condition leading 
to $\tilchi(\lamuf) = 1$ need not be satisfied at the end point of the range 
of possible durations.

The quadratic incentive, $\tilC_{\rm q}(s) = \half \lp 1 - s^2 \rp$ has the piecewise smooth
scalar incentive potential
\[
\tilchi_{\rm q}(s) = \lcb \begin{array}{ll}
\half \lp 1 + s^2  \rp & \qquad 0 < s \leq 1\medskip \\
s & \qquad s > 1
\end{array} \right . .
\]
It follows that if $0 < \lamu_0 \leq 1$, then (\ref{first_eq}) takes the form 
\[
r_{\rm q}(\lamu) = \frac {\lamu_0^3} 6 \lp 2 - \frac {\lamu}{\lamu_0} \rp \lp 1 +  \frac {\lamu}{\lamu_0} \rp^2
\sands
\tfin = \sqrt{\frac 6 {\lamu_0}} ,
\]
and hence
\beq{xq1}
x_{\rm q}(t; \tfin) = \lp \frac {t}{\tfin} \rp^2 \lp 3 - 2 \, \frac {t}{\tfin} \rp, 
\qquad \qquad \tfin \geq \sqrt{6}.
\eeq
The arbitrary duration synthesis problem solution is given by $\lamu_0 = 1$;
note  that $\tfin = \sqrt{6}$ is the minimum duration for the smooth $\tilC_{\rm q}$ solutions. 
(If $\lamu_0 > 1$, the corresponding solutions of the synthesis problem are 
piecewise smooth; we do not construct these solutions here.)

The total cost for the trajectory (\ref{xq1}) is
\[
\mbox{cost}(\tfin) = \int_0^{\tfin} (1 - \tilC_{\rm q}(\ddot x(t; \tfin))) dt 
=  \half \int_0^{\tfin} \lp 1 + \ddot x(t; \tfin)^2 \rp dt 
= \frac {\tfin} 2 + \frac 6 {\tfin^3}
\]
if $\tfin \geq \sqrt{6}$ (and hence $\lamu_0 \leq 1$). 
Note that if we replace the moderated cost function 
$1 - \tilC_{\rm q}(u) = \half \lp 1 + u^2 \rp$ with the kinetic
energy-style cost function $\frac {u^2} 2$, the evolution equation is unchanged,
but the total cost $\frac 6 {\tfin^3}$ is now a strictly decreasing function of 
the maneuver duration $\tfin$; hence none of the solutions of
this synthesis problem are solutions of the arbitrary duration problem. 
The convention that the moderation incentive equal zero on the boundary of the
admissible control region determines the constant that `selects' the fastest
smooth solution of the evolution equation as the solution of the arbitrary
duration problem.

\begin{figure}[t]
\begin{center}
\includegraphics[width=3in]{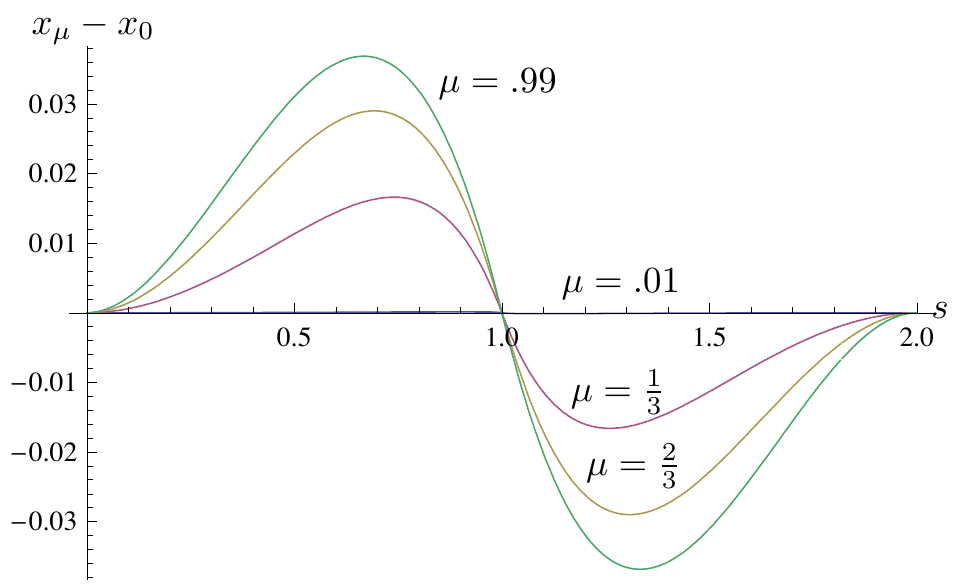} 
\includegraphics[width=3in]{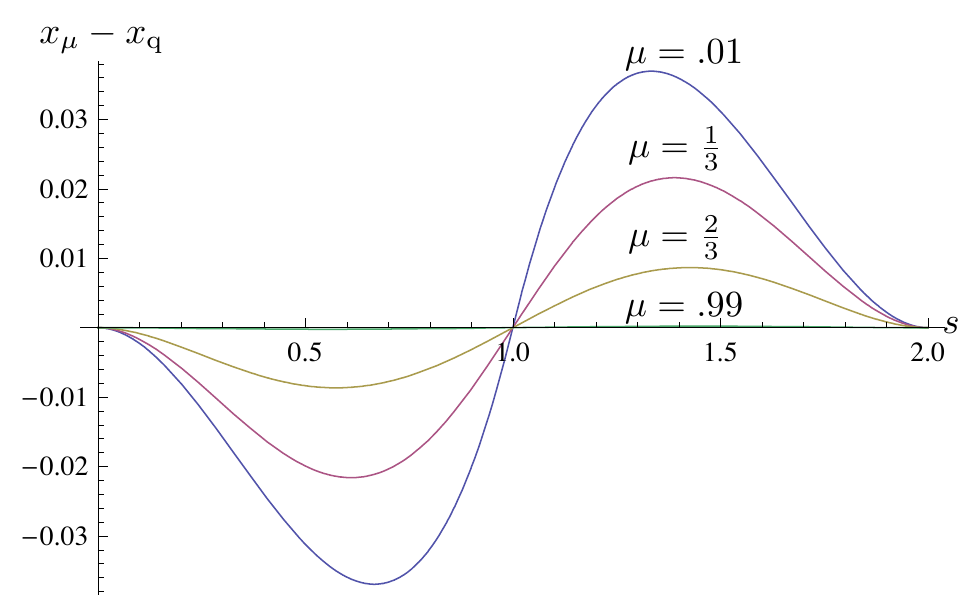} \smallskip \\
\includegraphics[width=3in]{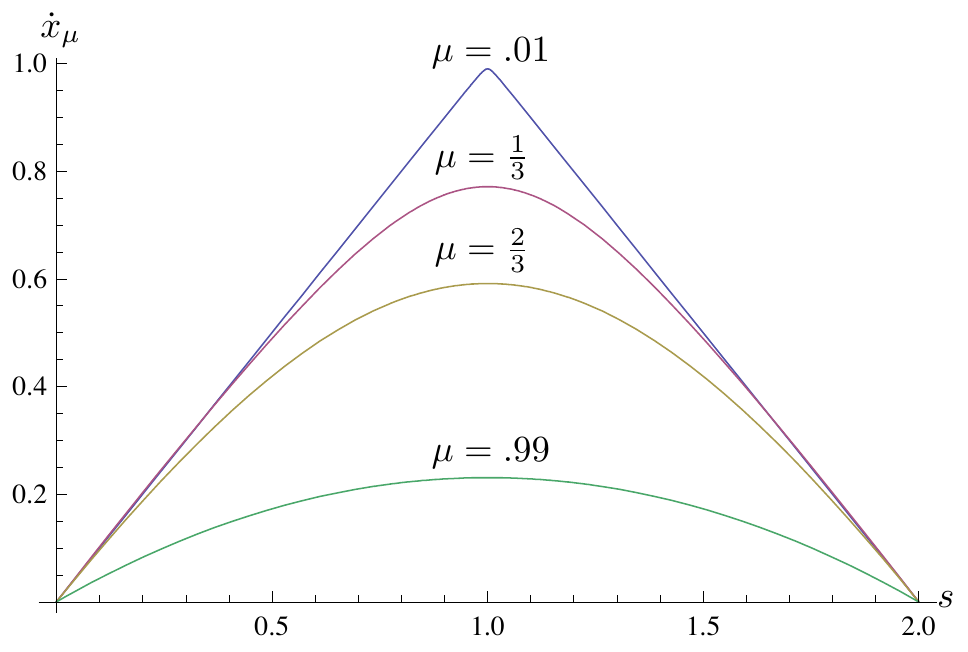} 
\includegraphics[width=3in]{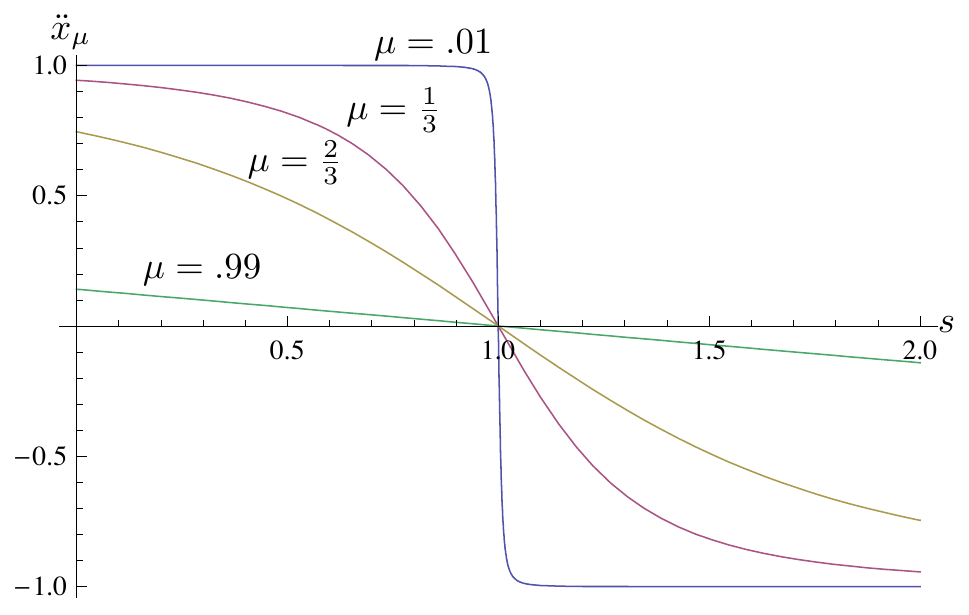} 
\end{center}
\caption{\label{tm_nc} Elliptical moderation incentive (EMI) solutions, plotted with respect to the rescaled time
$s := \frac {2 \, t}{\tfin(\mu)}$ : $\mu = 10^{-2}$, $\frac 1 3$, $\frac 2 3$, $1 - 10^{-2}$.
Upper left: differences between EMI solutions $x_\mu(s \, \tfin(\mu))$ and the time minimizing solution
$x_0(s)$; upper right: differences between EMI solutions $x_\mu(s \, \tfin(\mu))$ 
and the rescaled time quadratic incentive solution $x_{\rm q}(s)$;
lower left: velocities $\dot x_\mu(s \, \tfin(\mu))$;
lower right: accelerations: $\ddot x_\mu(s \, \tfin(\mu))$.}
\end{figure}

We now turn to the elliptic moderation incentives  
$\tilC_\mu(s) = \mu \sqrt{1 - s^2}$,  $0 < \mu \leq 1$, with
$\tilchi_\mu(s) = \sqrt{\mu^2 + s^2}$. (Note that $\tilC_0 = \tilC_{\rm tm}$.)
If we let
\[
g_\mu(\lamu) := - \half \lp  \lamu \, \chi_\mu(\lamu) + \mu^2 \ln \lp \chi_\mu(\lamu) + \lamu \rp \rp
\]
denote an anti-derivative of $\tilchi_\mu$, then
\beqa
r_\mu(\lamu) &=& \chi_\mu(\lamuf)(\lamu + \lamuf) + g_\mu(-\lamuf) - g_\mu(\lamu) \\
&=& \half \lp \chi_\mu(\lamuf) (\lamu + \lamuf) + (\chi_\mu(\lamuf) - \chi_\mu(\lamu)) \lamu
	- \mu^2 \ln \lp \frac {\chi_\mu(\lamu) + \lamu}{\chi_\mu(\lamuf) - \lamuf} \rp \rp.
\eeqa
The expression for $r_\mu$ simplifies somewhat for the solution to the arbitrary duration problem, 
for which $\chi_\mu(\lamuf) = 1$ and hence
\beq{rmu}
r_\mu(\lamu) = 
\half \lp \lamu + \sqrt{1 - \mu^2} + (1 - \chi_\mu(\lamu)) \lamu
	+ \mu^2 \ln \lp \frac {1 - \sqrt{1 - \mu^2}}{\chi_\mu(\lamu) + \lamu} \rp \rp.
\eeq
In this case, 
\beq{rmuf}
r_\mu(\lamuf) 
= \sqrt{1 - \mu^2} + \mu^2 \ln \frac {\mu}{1 + \sqrt{1 - \mu^2}}.
\eeq
 
We plot some information about  the solutions $x_\mu$ for some sample values of $\mu$ in 
Figure \ref{tm_nc}. To facilitate comparison, we plot $x_\mu$ and its derivatives using the 
scaled time $s := \frac {2 \, t}{\tfin(\mu)}$. Figure \ref{EMI_cost_dur} shows the total cost and
duration for the solutions of the arbitrary time synthesis problem, plotted as functions of the
moderation parameter $\mu$. 

As the moderation parameter goes to zero, the EMI control problem approaches the time minimization
one: 
\[
\lim_{\mu \to 0}(\tilC_\mu,  \sigma_\mu, x_\mu) = (\tilC_{\rm tm},  \sigma_{\rm tm}, x_{\rm tm}).
\]
In the limit $\mu \to 1$, the solutions $x_\mu$ of the arbitrary duration EMI problem
approach those of equal duration of the quadratic incentive synthesis problem. Specifically, 
\[
\lim_{\mu \to 1} \frac {x_\mu(t)}{x_{\rm q}(t; \tfin(\mu))} = 1 \qquad \mbox{for \quad $0 < t \leq \tfin(\mu)$},
\]
while $\lamuf = \sqrt{1 - \mu^2}$ and (\ref{rmuf}) imply that
\[
\lim_{\mu \to 1} \tfin(\mu)^2 \sqrt{1 - \mu^2} = 
\lim_{\mu \to 1} \frac {4 \, \lamuf^3}{r_\mu(\lamuf)}  
= 6.
\] 
(Recall that $\sqrt{6}$ is the duration of the solution of the arbitrary duration quadratic incentive 
problem.) As suggested by these limits and Figure \ref{tm_nc}, the family of EMI
solutions can regarded as linking those of the trivial and quadratic incentive problems. 

\begin{figure}[h]
\begin{center}
 \includegraphics[width=2.25in]{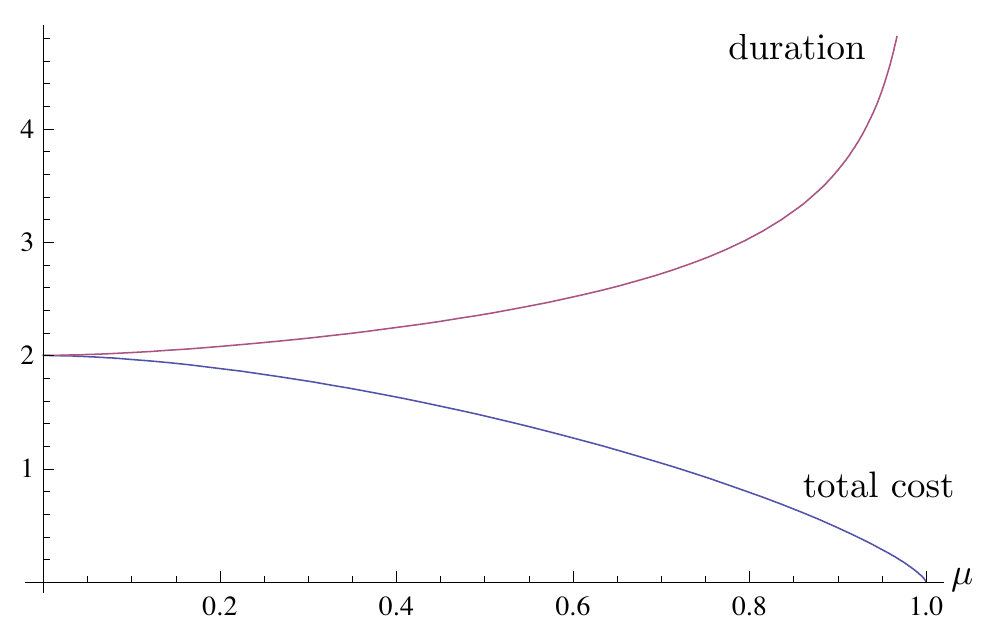} 
 \end{center}
\caption{\label{EMI_cost_dur}EMI solutions: duration, $\tfin(\mu)$, and total cost, $\mbox{cost}(\mu)$.}
\end{figure}

\rmk
For the EMI synthesis problem with $0 < \mu \leq 1$, the pair of second order ODEs 
(\ref{opt_gen_pair}) is equivalent to the fourth order ODE
\beq{spec_duration}
\mu \, \dddt \frac {\ddot x} {\sqrt{1 - \ddot x^2}} = - \nabla C(x);
\eeq
see Corollary \ref{double_cor}. 
In the present example, with constant $C$, the (nonzero) moderation parameter $\mu$ plays no role in 
(\ref{spec_duration}), and hence has no influence on the solution of the specified duration
problem. The parameter $\mu$ does, however,  effect  the solution of the arbitrary duration problem.
The analog of the condition $\tilchi(|\lamu_0|) = C(x_0)$ for an EMI controlled acceleration problem
with zero initial velocity is the initial acceleration condition
\[
\frac \mu {\sqrt{1 - \ddot x(0)^2}} = C(x_0),
\qquad \mbox{i.e.} \qquad
|\ddot x(0)| = \sqrt{C(x_0)^2 - \mu^2}.
\]
The solution to the specified duration problem with cost function $C_\mu$, $0 < \mu \leq 1$,
and duration $T > 2$ is $x_{\tfin^{-1}(T)}$. Thus the expressions (\ref{rmu}) and (\ref{rmuf}) are
sufficient to determine the solutions for both the specified and arbitrary duration problems; as in
the time minimization case, there is some redundancy in the $(x, \lamu)$ formulation.
\rmkend

\subsection{A one-dimensional controlled acceleration example: spooking}

We consider a generalization of the controlled acceleration example from Section \ref{warm_up},
adding a position penalty to the cost function; specifically, $C: \R \to [1, \infty)$,
with $C(1) = 1$ (the target is a `no-cost' position)  and $C(x) > 1$ for $x \in [0, 1)$.
This can be regarded as a very simple model of spooking---the reaction of, e.g.,
a grazing herbivore to unexpected abrupt noise or motion. When startled, the
animal will initially rush away from the disturbance; when it has reached its `comfort zone', 
it will either turn to examine the threat or resume grazing.  In our simple one-dimensional 
model, the function $C(x)$ acts as 
a `fear factor', modeling the undesirability of remaining near the perceived threat and 
providing an incentive to move rapidly to the comfort point $x = 1$.  

We consider the position-dependent cost function
\[
C(x) = 1 + \frac c 2 (1 - x)^2.
\]
The boundary conditions $x(0) = \dot x(0) = \dot x(\tfin) = 0$ and $x(\tfin) = 1$ imply
that a solution $(x, \lamu)$ of the arbitrary duration problem satisfies
\beq{spook_bcs}
\chi(\lamu_0) = C(0) = 1 + \frac c 2
\sands 
\chi(\lamuf) = 1.
\eeq
In addition, $\lamu$ must change sign at least once, since the boundary conditions for $x$
imply that $\sgn \ddot x = \sgn \lamu$ must change at least once. We assume that $\ddot x(0)$
is positive, and hence $\lamu_0 = \tilchi^{-1} \lp1 + \frac c 2 \rp$. 

\begin{figure}[t]
\includegraphics[height=1.75in]{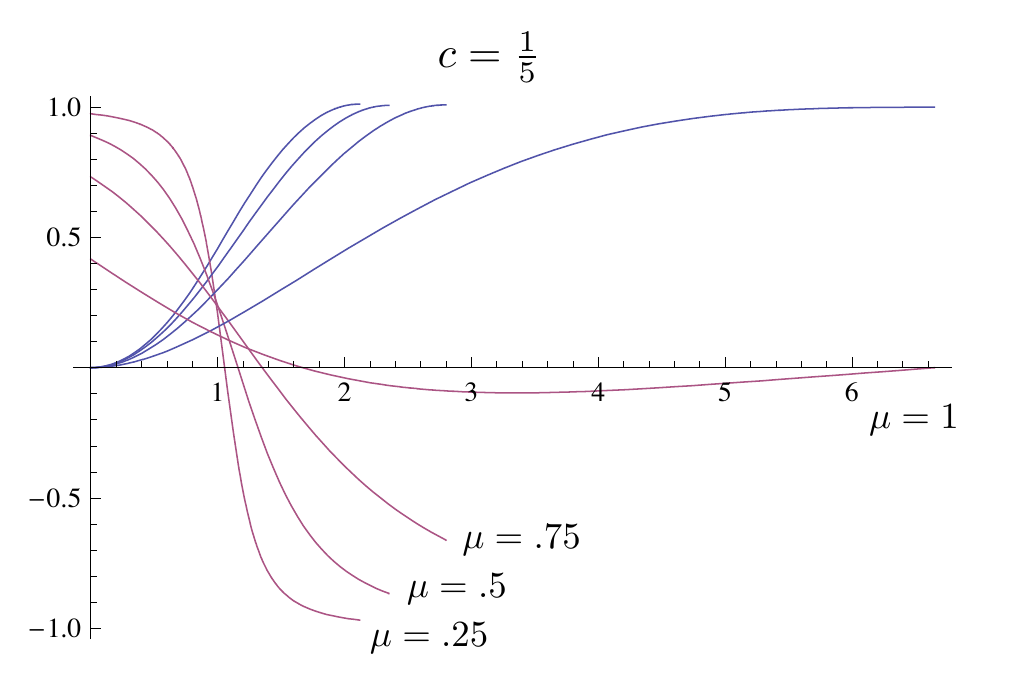} 
\includegraphics[height=1.75in]{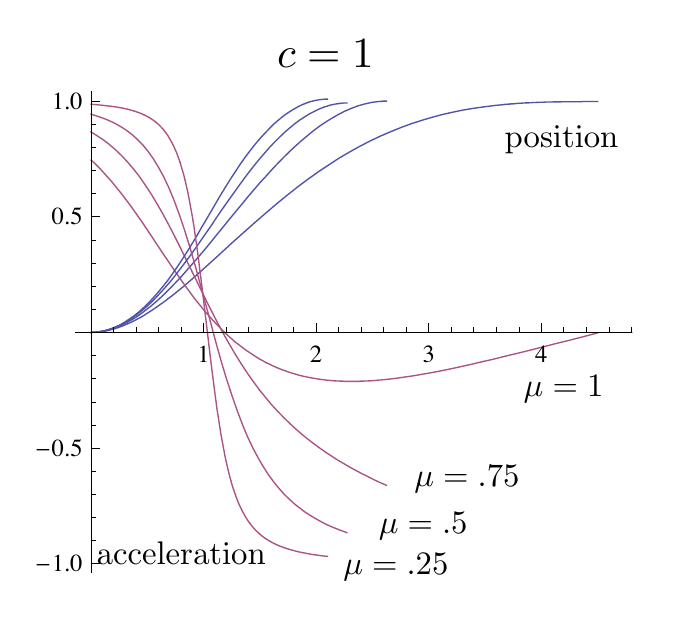} 
\includegraphics[height=1.75in]{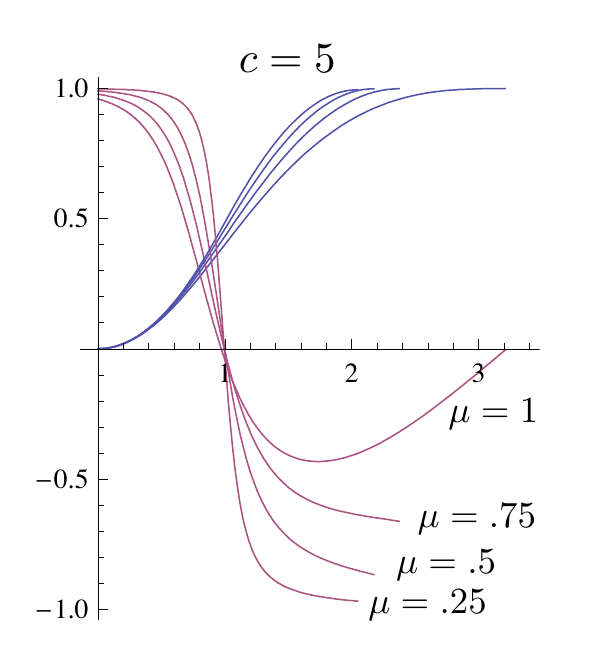} 
\caption{\label{spook_samples} Position and accleration plots for the solutions of the EMI synthesis
problem for $c = \frac 1 5$,  1, and 5, and $\mu = \frac 1 4$, $\frac 1 2$, 
$\frac 3 4$, and 1.}
\end{figure}

We first analyse the arbitrary duration synthesis problem for the moderation incentives $\tilC_\mu$.
Lacking closed form solutions for either the second order system 
\beq{second_spook}
\ddot x = \frac {\lamu}{\sqrt{\mu^2 + \lamu^2}}
\sands
\ddot \lamu = c \, (1 - x).
\eeq
derived from (\ref{opt_gen_pair}) for $C(x) = 1 + \frac c 2 (1 - x)^2$ and $\tilC_\mu$ or 
the fourth order ODE 
\beq{fourth_spook}
\mu \, \smallfrac {d^2 \ }{d t^2} \frac {\ddot x} {\sqrt{1 - \ddot x^2}} =  c (1 - x)
\eeq
derived from  (\ref{double_order}),  we turn to the reparametrized form of the system (\ref{reparam}), 
which can be implemented using standard numerical BVP routines. (Our numerical approximations 
were computed using the built-in {\sl Mathematica} function {\tt NDSolve}.) The boundary conditions on 
$\lamu$ take the form
\beq{boundary_lam}
\lamu_0 =  \lamu_0(c, \mu) := \sqrt{\lp 1 + \frac c 2 \rp^2 - \mu^2}
\sands
|\lamuf| = \sqrt{1 - \mu^2}.
\eeq
A solution of the reparametrized problem determines a solution of the synthesis problem
such that the auxillary function $\lamu$ has nonzero first derivative; $\lamu$ changes
sign at most once (and hence exactly once) in this situation. Hence we take 
$\lamuf = -  \sqrt{1 - \mu^2}$.  The reparametrized problem  takes the form of seeking
a solution $(r, q): [\lamu_0, \lamuf] \to \R \times \R^+$ of the BVP
\[
q \, r'  + 1 + \frac c 2 (1 - r)^2 = \sqrt{\mu^2 + \lamu^2}
\sands
q' = 2 \, c \, (1 - r),
\]
with boundary conditions $r(\lamu_0) = 0$ and $r(\lamuf) = 1$  for 
$\lamu_0 = \lamu_0(c, \mu)$ and $\lamuf = -  \sqrt{1 - \mu^2}$. This BVP can be solved
numerically (using, e.g., a shooting method).  Once a solution has been found, the 
elapsed time can be found as a function of $\lamu$ by numerically computing the integral
\[
t(\lamu) = \int_{\lamu_0}^\lamu \frac {d s}{\sqrt{q(s)}}.
\]
The desired solution $x$ is given by $x = r  \circ t^{-1}$.
Figure \ref{spook_samples} shows some sample plots, 
with $c = \frac 1 5$,  1, and 5, and $\mu = \frac 1 4$, $\frac 1 2$, 
$\frac 3 4$, and 1.

\rmk
We shall see that, as in the previous example, the total cost function for the QCC
problem is a monotonically decreasing function of the maneuver duration, with countably
many inflection points, corresponding to trajectories that oscillate about the target before
coming to rest. We do not rule out the possibility that such oscillations may also  lead to a 
decrease in cost for the elliptical moderation incentives; our numerical searches are directed 
only towards the identification of trajectories that do not overshoot the target.  
\rmkend

Note that  Figure \ref{spook_samples} shows little response to the position penalty strength $c$ for 
small values of $\mu$---the solutions remain close to the corresponding solutions for $c = 0$. 
In Figure \ref{spook_accel_diffs}
we plot $\ddot x_\mu(t; c) - \ddot x_\mu(t; 0)$ for $\mu = \frac 1 8$, $\frac 1 4$, and $\frac 1 2$, and
$c = \frac 1 5$ and 1. 
On the other hand, the solutions for $\mu$ near 1 are strongly influenced by the position 
penalty; the initial acceleration increases dramatically with $c$ and the acceleration curve changes
from concave to convex as $c$ increases. A sufficiently frightening event will provoke an initially strong
response even when the moderation incentive is high, but the later stages of the recovery from that 
initial response are largely determined by the moderation incentive. This roughly agrees with 
actual spooking behavior: when startled, an inexperienced animal will often respond by first 
bolting, then abruptly halting and
whirling about to examine the apparent threat while still in a state of high excitement; 
a more experienced one can still be spooked, but quickly regains its composure in the absence of 
real danger and gradually decelerates, reducing the significant skeletomuscular stresses of a hard 
stop and risk of self-inflicted injury during rapid motion. 

\begin{figure}[h]
\begin{center}
\includegraphics[width=2.75in]{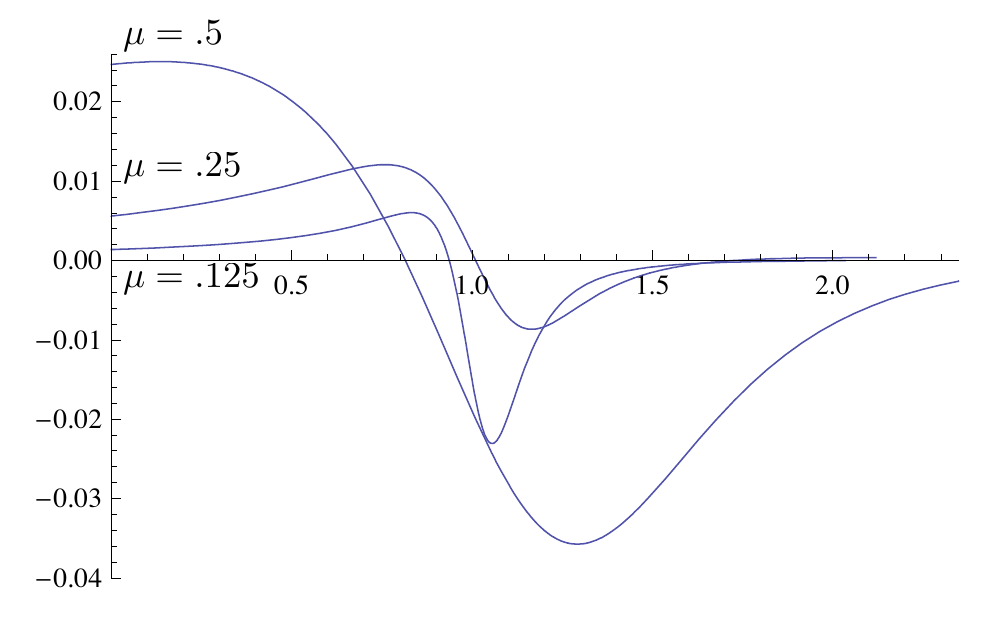} 
\qquad \includegraphics[width=2.75in]{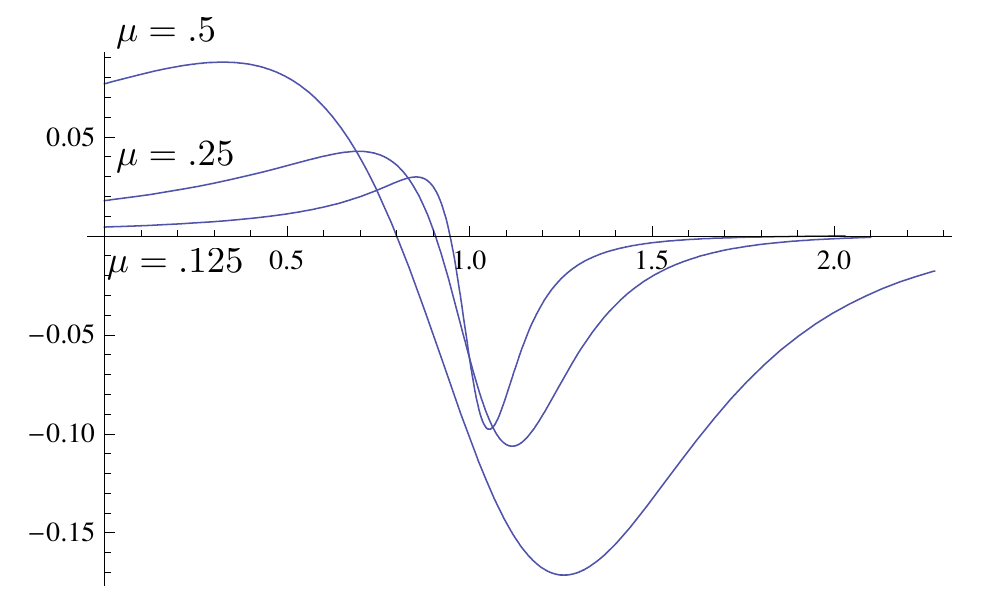}
\caption{\label{spook_accel_diffs}$\ddot x_\mu(t; c) - \ddot x_\mu(t; 0)$ for $\mu =  \frac 1 8$, $\frac 1 4$, and $\frac 1 2$.
left: $c = \frac 1 5$, right: $c = 1$. }
\end{center}
\end{figure}

The quadratic polynomial moderation incentive $\tilC_{\rm q}$ yields a piecewise smooth
solution of the arbitrary
time synthesis problem that is simpler in some regards, but less convenient in others, than the
solutions for the elliptical incentives. As we shall see, the controls for the 
$\tilC_{\rm q}$ solutions satisfying (\ref{opt_gen_pair}) are continuous, but not everywhere 
differentiable if the parameter $c$ in the position-dependent component of the cost function 
is nonzero; the solution behaves like the time-minimization solution, with control identically equal 
to one, for the first part of the manuever, then satisfies a linear fourth order ODE for the 
remainder of the manuever.
  
Equation (\ref{QCC_cp}) implies that the restriction of $\tilchi_{\rm q}^{-1}$ to $[1, \infty)$ is the identity  map; 
hence the initial condition $\chi_{\rm q}(\lamu_0) = C(0) = 1 + \frac c 2$ implies that 
$|\lamu_0| = 1 + \frac c 2$ and $|\ddot x(0)| = 1$. A similar argument shows that the terminal condition 
$\chi_{\rm q}(\lamuf) = C(1) = 1$ implies that $|\lamuf| = |\ddot x(\tfin)| = 1$.
Since $\dot x(0) = \dot x(\tfin)$ implies that
$\ddot x$, and hence $\lamu$ changes sign, $\lamu$ must pass through zero; hence there 
exists $t_* > 0$ such that $\lamu(t) \geq 1$ for $0 \leq t \leq t_*$ and $\lamu(t) < 1$ on some interval 
$(t_*, t_* + \epsilon)$. We find solutions to the arbitrary duration problem satisfying
$1> \lamu > -1$ on $(t_*, \tfin)$, with $\lamu(\tfin) = -1$. (As we shall show later, there are solutions 
to the specified duration synthesis problem that overshoot and oscillate about the target.)

The second order system (\ref{opt_gen_pair}) equals that for $\tilC_0$ when $|\lamu| > 1$:
$x(t)$ satisfies $\ddot x \equiv 1$ on the interval $[0, t_*]$; the initial conditions $x(0) = \dot x(0) = 0$
imply that $x(t) = \frac {t^2} 2 $ on $[0, t_*]$. On $(t_*, \tfin]$,
\[
\ddot x = \chi_{\rm q}'(\lamu) = \lamu,
\qquad \mbox{and hence} \qquad
x^{(4)} = \ddot \lamu =  c (1 - x).
\]
The linear fourth order PDE $y'''' + 4 \, y = 0$ has the general solution
\beq{QCC_accel}
y(s)  =  \lp \begin{array}{c} \cosh s \\ \sinh s \end{array} \rp^T
M \lp \begin{array}{c} \cos s \\ \sin s \end{array} \rp
\eeq
for an arbitrary matrix $M \in \R^{2 \times 2}$; (\ref{QCC_accel}) satisfies the initial conditions
$y(0) = y_0$, $y'(0) = v_0$, and $y''(0) = a_0$ if and only if
\beq{M_conds}
M = 
\lp \begin{array}{cc} y_0 & {\displaystyle \frac {v_0 + m} 2} \medskip \\
{\displaystyle \frac {v_0 - m} 2} & {\displaystyle \frac {a_0} 2} \end{array} \rp
\eeq
for some $m \in \R$.  

\begin{figure}
\begin{center}
\includegraphics[height=1.65in]{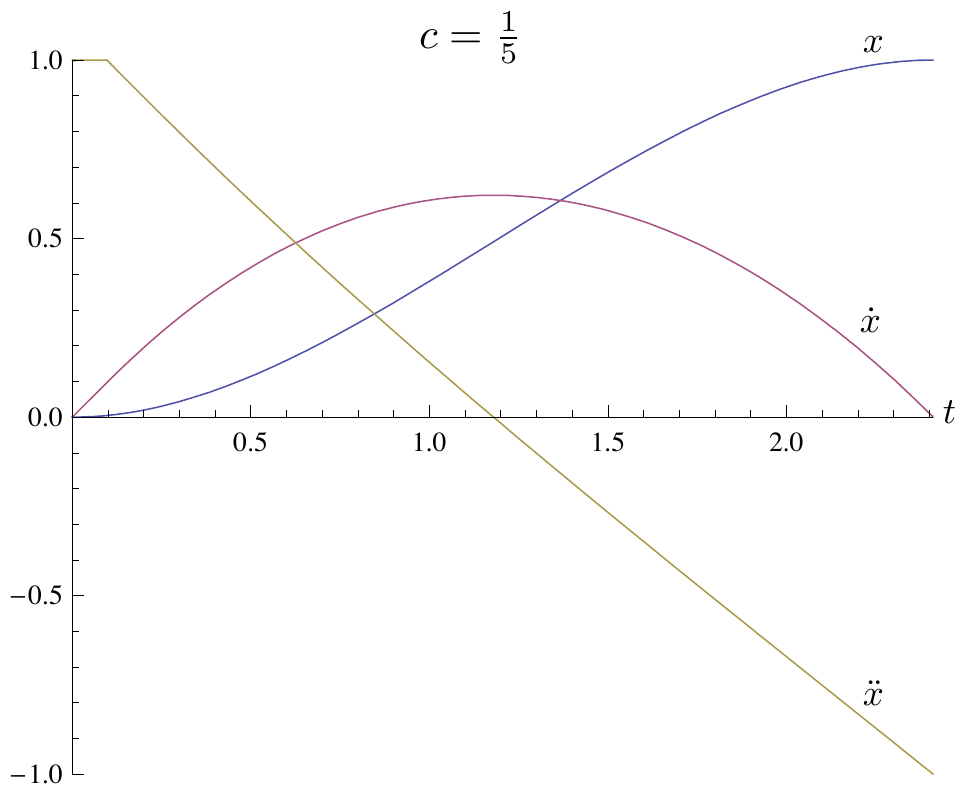} \quad
\includegraphics[height=1.65in]{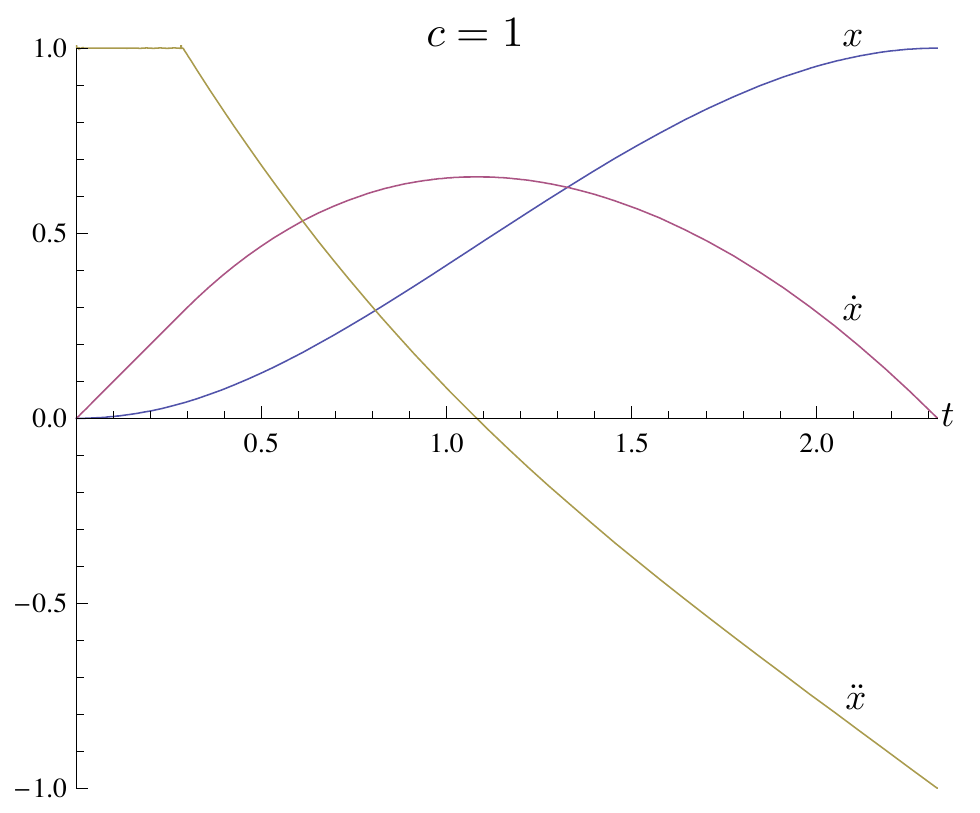} \quad
\includegraphics[height=1.65in]{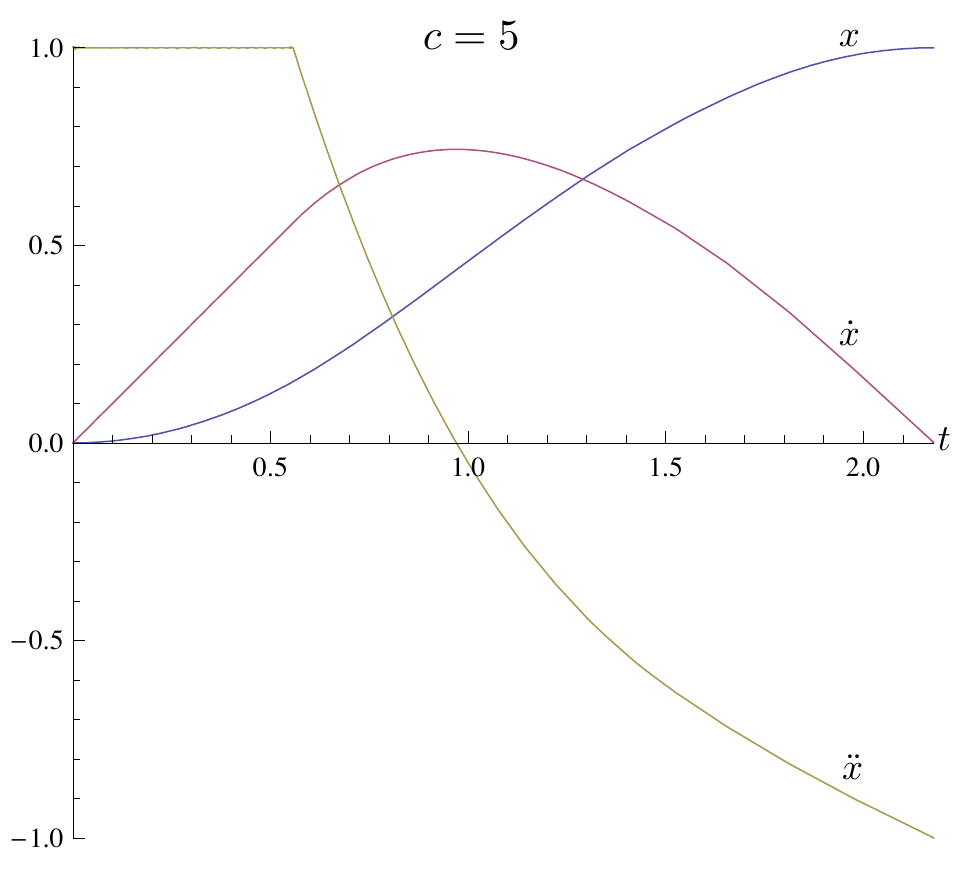} 
\caption{\label{spook_samples_qcc} Position, velocity, and acceleration plots for the solutions of the 
synthesis problem with moderation incentive $\tilC_{\rm q}$, $c = \frac 1 5$, 1, 5.}
\end{center}
\end{figure}

We seek $t_*$, $\tfin$, and $M$ such that 
\beq{hybrid}
x(t) = \lcb \begin{array}{ll} \frac {t^2} 2 & \qquad 0 \leq t \leq t_* \\
1 - y \lp c^{1/4}/\sqrt{2} \, (t - t_*)\rp & \qquad t_* < t \leq \tfin
\end{array} \right .
\eeq
is twice differentiable and satisfies the boundary conditions $x(\tfin) = 1$, $\dot x(\tfin) = 0$,
and $\ddot x(\tfin) = -1$. Hence we require
\[
y_0 = 1 - \frac {t_*^2} 2, \qquad v_0 = - t_*, \sands
a_0 = - \lp c^{-1/4} \sqrt{2} \rp^2 \ddot x(t_*) = - \frac 2 {\sqrt{c}}.
\]
If we set $\sfin := c^{1/4}/{\sqrt 2} \, (\tfin - t_*)$, then the terminal conditions are
\beq{end_conds}
y(\sfin) =  y'(\sfin) = 0 \sands
y''(\sfin) =  - \lp c^{-1/4} \sqrt{2} \rp^2 \ddot x(\tfin) = \frac 2 {\sqrt{c}}.
\eeq
Solving (\ref{end_conds}) for $y_0$, $v_0$, and $m$, given $a_0 = - \frac 2 {\sqrt{c}}$,
yields
\beqa
\sqrt{c} \lp \begin{array}{c} y_0 \\ v_0 \\ m \end{array} \rp 
&=& \lp \begin{array}{c} y_0(\sfin) \\ v_0(\sfin) \\ m(\sfin) \end{array} \rp\\
&:=& \frac 1 {\cosh \sfin  \, \sin \sfin + \cos \sfin \, \sinh \sfin} 
 \lp \begin{array}{c} (\cosh \sfin + \cos \sfin)(\sinh \sfin - \sin \sfin) \\
- (\sinh \sfin - \sin \sfin)^2 \\ (\cosh \sfin + \cos \sfin)^2
\end{array} \rp.
\eeqa
The matching condition 
\[
\frac {v_0(\sfin)^2} {4 \, \sqrt{c}} = \frac {\dot x(t_*)^2} 2 
= \frac {t_*^2} 2 = x(t_*) = 1 - \frac {y_0(\sfin)}{\sqrt{c}},
\]
equivalently
\[
\sqrt{c} = y_0(\sfin) + \frac {v_0(\sfin)^2} 4,
\]
implicitly determines a function 
$\sfin: \R^+ \to (0, \tilde \sfin)$, where $\tan \tilde \sfin + \tanh \tilde \sfin = 0$. The functions
\[
t_*(c) := - \frac {v_0(\sfin(c))}{\sqrt{2} \, c^{1/4}}
\sands
\tfin(c) := \frac {\sqrt{2} \, \sfin(c)}{c^{1/4}} + t_*(c) 
= \frac {\sqrt{2}}{c^{1/4}} \lp \sfin(c) - \frac {v_0(\sfin(c))} 2 \rp
\]
give the desired transition time and duration. The curves $t_*(c)$ and $\tfin(c)$ are shown in 
Figure \ref{qcc_times}; plots of the position, velocity, and acceleration for some representative 
values of $c$ are shown in Figure \ref{spook_samples_qcc}.

\begin{figure}
\begin{center}
\includegraphics[width=2.25in]{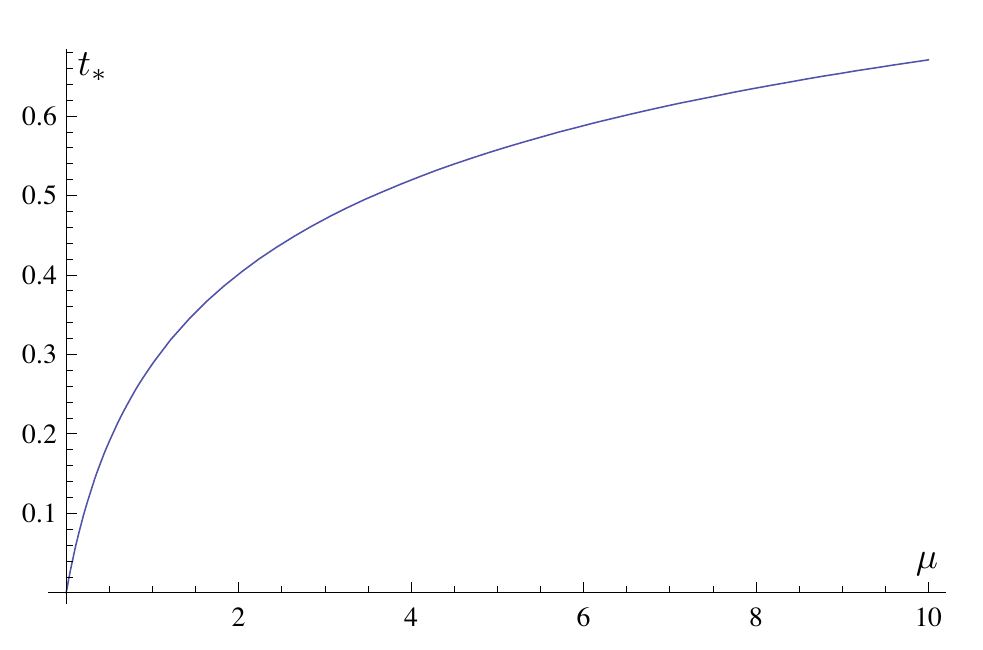} 
\qquad \includegraphics[width=2.25in]{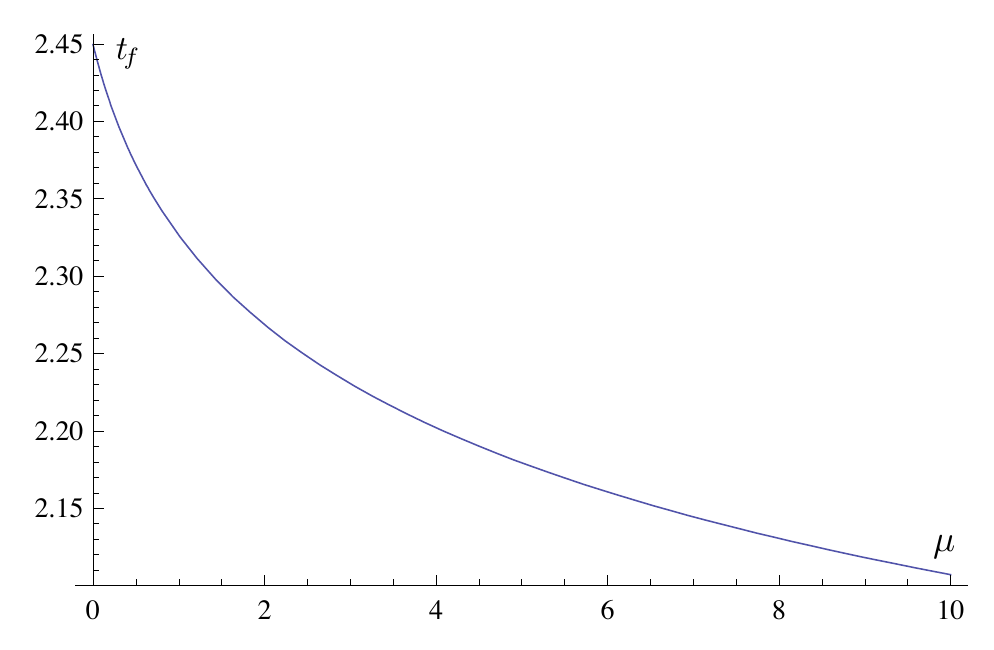}
\caption{\label{qcc_times}Transition time $t_*$ and total duration $\tfin$
for the quadratic moderation incentive $\tilC_{\rm q}$. }
\end{center}
\end{figure}

We now briefly consider the alternative cost function 
$\widehat C_{\rm qcc}(x, u) = \frac {u^2} 2 + \frac c 2 (1 - x)^2$. This cost function, in which the
term $\frac {u^2} 2$ is naturally interpreted as a positive control cost added to the position-dependent
cost, differs from the $\tilC_{\rm q}$ moderated problem analysed above only by the constant $\half$. Thus solutions of the synthesis problem for one cost function are solutions for the other, as
well.  A solution of the  synthesis problem for the moderation incentive
$\tilC_{\rm q}$ is a solution for the arbitrary duration synthesis problem for $\widehat C_{\rm qcc}$ 
if and only if the Hamiltonian determined by  $\widehat C_{\rm qcc}$ is equal to zero along the solution;
equivalently: if and only if $\ddot x(0)^2 = c$ and $\ddot x(\tfin) = 0$. If we restrict our attention to
$c \leq 1$, then this condition on the initial acceleration is compatible  with the general 
control constraint $|\ddot x| \leq 1$.  We seek $M \in \R^{2 \times 2}$ and $\sfin \in \R^+$ such that 
$y(s)$ given by (\ref{QCC_accel}) satisfies the boundary conditions $y(0) = 1$, 
$y'(0) = y(\sfin) = y'(\sfin) = y''(\sfin) = 0$, and $y''(0) = -2$, and hence
\[
x(t) = 1 - y \lp c^{1/4}/\sqrt{2} \, t \rp 
\]
satisfies the boundary conditions of the arbitrary duration synthesis problem for $\widehat C_{\rm qcc}$,
where $\tfin = c^{-1/4} \, \sqrt{2} \, \sfin$, as before. After substituting $y_0 = 1$, $v_0 = 0$, and 
$a_0 = - 2$ into (\ref{M_conds}), we find that the terminal conditions $y(\sfin) = y'(\sfin) = y''(\sfin) = 0$ 
are satisfied if and only if $m = 2 \, \coth \sfin$ and $\sin \sfin = 0$. Thus the solutions of the arbitrary
duration synthesis problem for $\widehat C_{\rm qcc}$ have duration $\tfin(c, k) := c^{-1/4} \, \sqrt{2}\, \pi \, k$, 
$k \in {\mathbb N}$.  If $k > 1 $, the solution of duration  $\tfin(c, k)$ overshoots the destination,
making $\frac {k - 1} 2$ oscillations about the target before stopping.

\begin{figure}
\begin{center}
\includegraphics[width=2.75in]{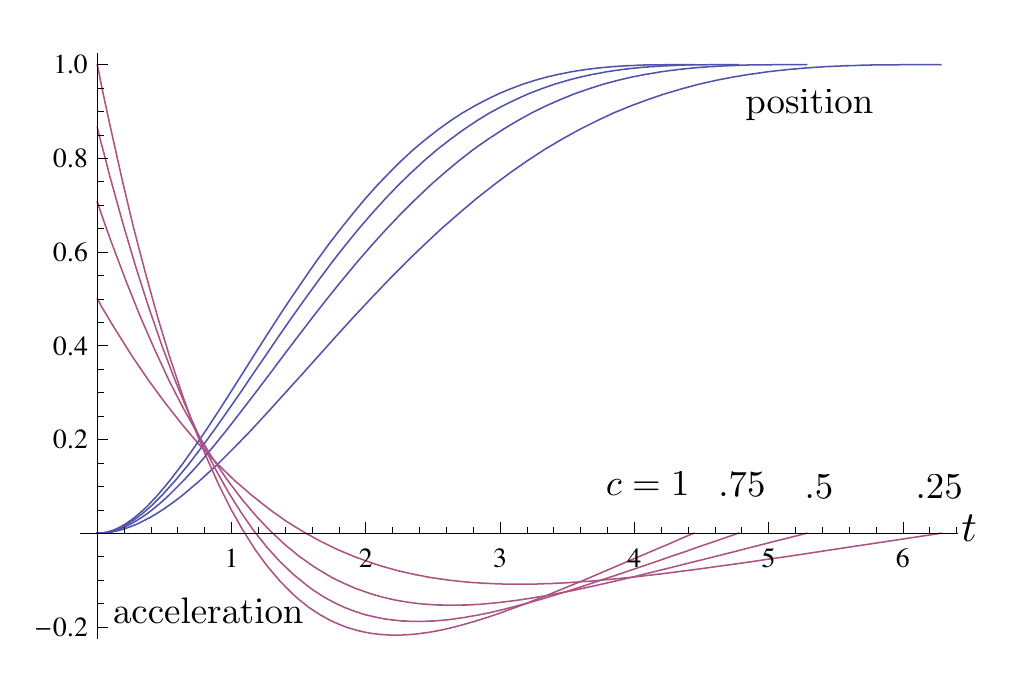} 
\caption{\label{qcc_sols}Solutions of the arbitrary duration synthesis problem for $\widehat C_{\rm qcc}$,
for $c = .25$, $.5$, $.75$, $1$ and $k = 1$. }
\end{center}
\end{figure}

\begin{figure}[h]
\begin{center}
\includegraphics[width=2.25in]{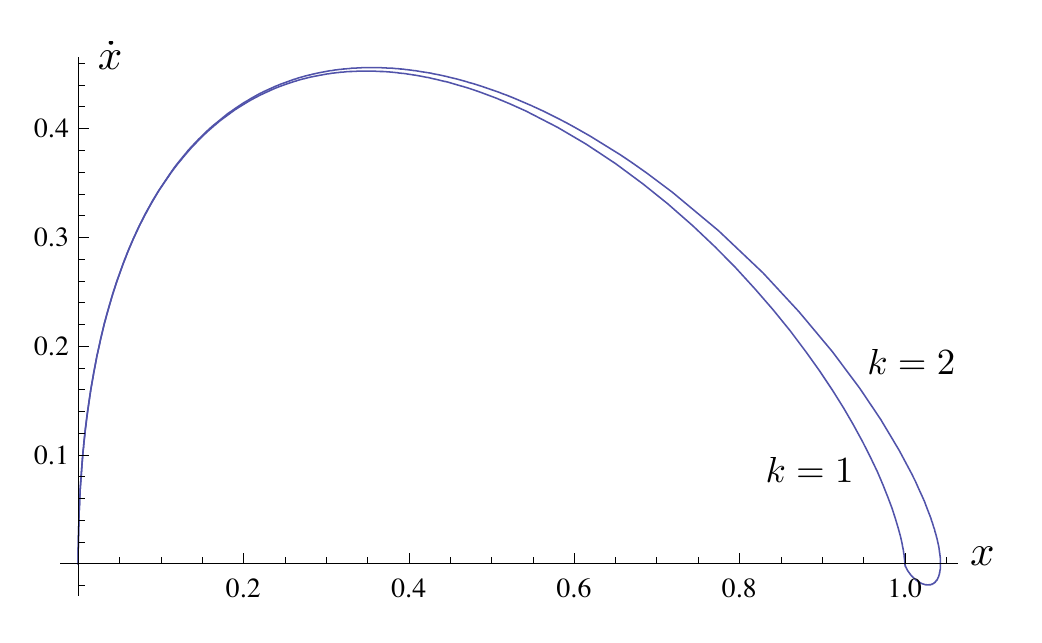} 
\qquad \includegraphics[width=2.25in]{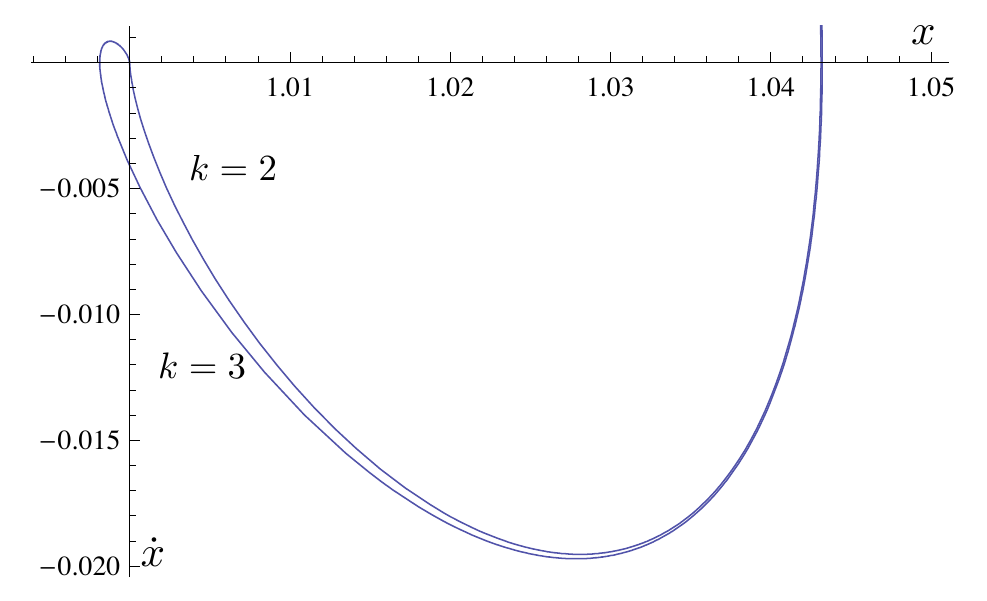}
\caption{\label{qcc_overshoot}Parametric plots $(x, \dot x)$ of QCC solutions for $c = 1$, 
$\tfin = k \, \sqrt{2} \, \pi$. Left: $k = 1$ and $2$; right: close-up of $k = 2$ and $3$}
\end{center}
\end{figure}
We now  show that none of these solutions of the arbitrary duration problem for $\widehat C_{\rm qcc}$
actually minimize the total cost;
\[
\widehat C_{\rm tot}(\tfin) :=  \int_0^{\tfin} \widehat C_{\rm qcc}(x, \ddot x) dt 
= \frac {c \, m} 2,
\qquad \mbox{with derivative} \qquad
\widetilde C_{\rm tot}'(\tfin) 
=  - 2 \, c \, \ddot x(\tfin)^2.
\]
Thus the cost is nonincreasing and the solutions of duration $\tfin(k)$ are all inflection points,
satisfying $\widehat C_{\rm tot}(\tfin(c, k)) =  2 \, c \, \coth k \, \pi$.
Hence increases in the total maneuver time yield exponentially small reductions in cost. 

\begin{figure}
\begin{center}
\includegraphics[width=2.5in]{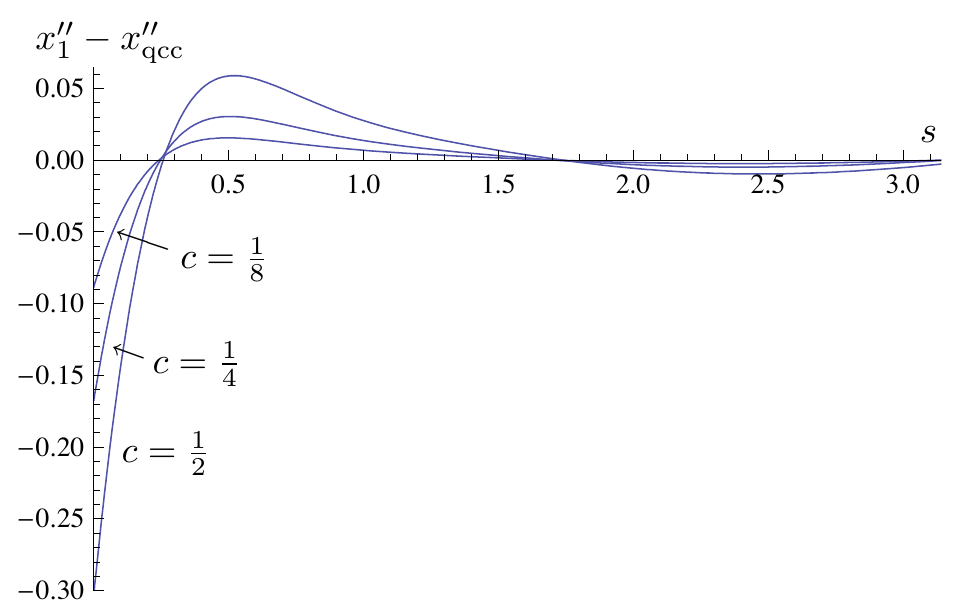} 
\caption{\label{spook_qcc_diffs}
Differences in acceleration for rescaled time trajectories:
$x_1''(s; c) - x_{\rm qcc}''(s; c)$ for $c = \frac 1 8$, $\frac 1 4$, and $\frac 1 2$.}
\end{center}
\end{figure}
For relatively small values of $c$, Figures \ref{spook_samples} and \ref{qcc_sols} show a qualitative 
resemblance between the moderated solutions for $\tilC_1$ and the solutions for $\widehat C_{\rm qcc}$.
To facilitate the comparison of these solutions, we reparametrize the  $\tilC_1$ solutions 
using  the rescaled time 
$s := \smallfrac {c^{1/4}} {{\sqrt 2}} \, t$. The rescaling of the fourth evolution equation 
(\ref{fourth_spook}) takes the form
\[
\mu \, \smallfrac {d^2 \ }{d s^2} \frac {x''} {\sqrt{1 - \smallfrac c 4 (x'')^2}} = 4 (1 - x).
\]
Hence the moderated solution approaches that of the QCC problem as $\mu \to 1$ and 
$c \to 0$, but the two families are not equal for nonzero $c$. See Figure \ref{spook_qcc_diffs} 
for a comparison of the controls for the two families for some representative values of $c$.

\section{Conclusions and future work}

Our choices of state space, vector fields, admissible control regions, and unmoderated cost
functions were intended to be the simplest possible. We intend to 
generalize each of these components of the synthesis problem. 

Optimal control on nonlinear manifolds has received significant attention in recent years,
particularly situations in which the controls can be modeled as elements of a distribution
within the tangent bundle of the state manifold, corresponding to (partially) controlled 
velocities. See, e.g. \cite{NVdS, Sontaga, Sontagb, Bloch}, and references therein. Analogous
constructions for (partially) controlled higher order
derivatives (e.g. controlled acceleration) can be implemented using jet bundles, but can
be unwieldy in practical implementations. We are particularly interested in Lie groups, 
since these manifolds possess additional structure that facilitates the identification of the 
controls with elements of a single vector space, the Lie algebra. Results for conservative
systems on Lie groups, homogeneous manifolds, and associated bundles should be easily 
extended to optimal control problems.  Geometric integration schemes for the numerical
integration of Hamiltonian systems can be used to approximate the
solutions of synthesis problems on such manifolds; see \cite{LS, Leimkuhler, 
synode, LO1, LO} and references therein.

The conservation law (\ref{cons_law}) can play a crucial role in the exact or approximate 
solution of the synthesis problem. In \cite{Lewis_cons} we develop several results, including
a reduction of the evolution equation for the auxiliary variables to the sphere, that exploit
the conservation law. We show that a simple vertical take-off interception model with controlled
velocities can be reduced to quadratures using this approach. We intend to generalize these results 
and apply them to more complex systems in future work. 

The assumption that the admissible control set is the unit ball can be 
relaxed to  $\lcb u \in V : f(u) \leq  c \rcb $ for some vector space $V$,
differentiable function $f: V \to [f_{\rm min}, \infty)$  with $\nabla f$ everywhere nonzero on 
$\partial \cU$, and constant $c \in \R$. The analog of $\tilC_{\rm q}$ would be 
$\tilC(u) = c - f(u)$; the analogs of $\tilC_\mu$ would be $\mu \, \sqrt{c - f(u)}$. Clearly the
property that the control $u$ would be a rescaling of the auxiliary variable $\lamu$ would not
hold; hence the determination of the value of $u$ maximizing the Hamiltonian would, in 
general, be  more complicated than in the case considered here.
More generally, the incentive could be a function of both control
and state variables, retaining the property that the function rewards avoidance
of the boundary of the admissible control set. 

We intend to investigate the skewed gradient equations (\ref{opt_gen_pair}) in greater detail,
seeking both analytic properties and efficient numerical schemes. We hope to model
various biomechanical systems using optimal control formulations of the type described here,
investigating the utility of moderation incentives in the interpretation of animal motion and
behavior.  The construction of families of optimal solutions parametrized by moderation or 
urgency may shed light on aspects of motion planning that are not easily understood using a
single cost function. For example, most of us utilize a range of strategies when lifting and carrying
objects: a newborn, a laptop, and a phone book merit different levels of caution. 
Significant expenditure of energy is required to capture prey or evade a predator, but 
exhaustion and lameness leave both predator and prey vulnerable to future attacks and
starvation---long-term survival depends on the adjustment of resource consumption to the demands 
of each encounter. Even very simple mathematical models can add to our understanding 
of the adaptability of natural control systems.

\eject
\bibliography{control}
\bibliographystyle{siam}
\end{document}